\documentclass[reqno, 10pt]{article}
\usepackage{enumerate,amsmath,amssymb,bm,ascmac,amsthm,url,enumerate}
\usepackage{tikz}
\usepackage[margin=30truemm]{geometry}
\usepackage{here}
\usepackage {mnotes} %Revise用コメントで必要
\usepackage {comment} %つけたした

\usepackage[shortlabels]{enumitem}
\setlist[enumerate,1]{label={\upshape(\roman*)}}

\newtheorem{thm}{Theorem}[section]
\newtheorem{thm*}{Theorem}
\newtheorem{cor}[thm]{Corollary}
\newtheorem{lem}[thm]{Lemma}
\newtheorem{lem*}{Lemma}

\newtheorem{pro*}{Proposition}
\theoremstyle{definition}
\newtheorem{defi}[thm]{Definition}
\newtheorem{defi*}{Definition}
\newtheorem{ex}[thm]{Example}

\newcommand{\MC}[1]{\mathcal{#1}}
\newcommand{\MB}[1]{\mathbb{#1}}

\newcommand{\C}[2]{\begin{pmatrix} #1 \\ #2 \end{pmatrix}}

\newcommand{\TR}[1]{}
\newcommand{\TB}{\textcolor{black}}
\newcommand{\G}{\Gamma}

\newcommand{\Span}[1]{\left \langle #1 \right \rangle}

\newcommand{\e}{\epsilon}
\newcommand{\spec}{\sigma}
\newcommand{\aut}{g}

\DeclareMathOperator{\Aut}{Aut}

\renewcommand{\MNOTE}[2][]{}

\allowdisplaybreaks

\makeatletter

\@addtoreset{equation}{section}
\makeatother

\title
{
Circulant graphs with valency up to 4 that admit perfect state transfer in Grover walks
\MNOTE{Change A.(ii) of Major comments}
}

\author{
Sho Kubota\thanks{
Department of Mathematics Education,
Aichi University of Education,
1 Hirosawa, Igaya-cho, Kariya, Aichi 448-8542, Japan.
\texttt{skubota@auecc.aichi-edu.ac.jp}}
\and
Kiyoto Yoshino\thanks{
Department of Computer Science,
Faculty of Applied Information Science,
Hiroshima Institute of Technology,
Saeki Ward, Hiroshima, 731-5143, Japan.
\texttt{k.yoshino.n9@cc.it-hiroshima.ac.jp}
}
}
\date{}

\begin{document}
\maketitle
\begin{abstract} \MNOTE{Change A.(iii) of Major comments}
We completely characterize circulant graphs with valency up to $4$ that admit perfect state transfer. 
Those of valency $3$ do not admit it.
On the other hand, circulant graphs with valency $4$ admit perfect state transfer only in two infinite families: one discovered by Zhan and another new family, while no others do.
The main tools for deriving these results are symmetry of graphs and eigenvalues.
We describe necessary conditions for perfect state transfer to occur based on symmetry of graphs,
which mathematically refers to automorphisms of graphs.
As for eigenvalues, if perfect state transfer occurs,
then certain eigenvalues of the corresponding isotropic random walks must be the halves of algebraic integers.
Taking this into account,
we utilize known results on the rings of integers of cyclotomic fields.
\vspace{8pt} \\
{\it Keywords:} perfect state transfer, Grover walk, circulant graph, automorphism \\
{\it MSC 2020 subject classifications:} 05C50; 81Q99
%05C50: Graphs and linear algebra
%05C20 Digraphs
%05C81: Random walks on graphs
%81Q99: None of them (quantum theory), but in this section
\end{abstract}

\section{Introduction} \label{s1}
Quantum walks are a powerful tool in quantum computing and quantum information processing,
with potential applications in quantum algorithms~\cite{ambainis2007quantum},
quantum simulations~\cite{de2014quantum},
and quantum cryptography~\cite{vlachou2015quantum}.
They have also been studied from a purely mathematical point of view as they show interesting properties not found in classical random walks, such as localization \cite{inui2004localization} and periodicity \cite{higuchi2017periodicity}.
One of the most attractive properties of quantum walks is perfect state transfer,
which roughly speaking refers to transferring between specific quantum states with probability $1$.
Perfect state transfer has important applications in quantum communication and quantum information processing \cite{christandl2004perfect}.
%Quantum states involved in the transfer may contain some kind of message, or they may be used to create entanglement for quantum teleportation between two sites \TR{(Citation)}.
%it can be used to establish entanglements between remote nodes or to transmit quantum information over long distances.

The subject of this paper is perfect state transfer in Grover walks, which are typical discrete time quantum walks.
Although perfect state transfer in continuous time quantum walks has been well studied,
those in discrete time quantum walks have not been as well investigated as in continuous time quantum walks.
For literature on perfect state transfer in continuous time quantum walks, we refer to Godsil's survey \cite{godsil2012state}.
Representative studies of perfect state transfer in discrete time quantum walks are summarized in Section 1 of Zhan's paper \cite{zhan2019infinite}.
Among these, the studies dealt with Grover walks are summarized in Table~\ref{0316-1}.

\begin{table}[h]
  \centering
  \begin{tabular}{|c|c|}
\hline
Graphs & Ref. \\
\hline
\hline
Diamond chains & \cite{kendon2011perfect} \\ \hline
$\overline{K_2} + \overline{K_n}$,
$\overline{K_2} + C_n$ & \cite{barr2012periodicity} \\ \hline
Circulant graphs with valency $4$ & \cite{zhan2019infinite} \\ \hline
%Complete multipartite graphs with partite sets having the same size
$K_{m, m, \dots, m}$ & \cite{kubota2022perfect}  \\ \hline
\end{tabular}
\caption{Previous works on perfect state transfer in Grover walks} \label{0316-1}
\end{table}

We focus on Zhan's work \cite{zhan2019infinite}.
Zhan discovered an infinite family of circulant graphs with valency $4$ that admits perfect state transfer.
We extend her result and completely characterize circulant graphs up to valency $4$ that admit perfect state transfer.
In order to describe our results in detail,
we will introduce the symbols and terms used in this paper

%See \cite{godsil2013} for basic terminologies related to graphs.
Let $\G =(V, E)$ be a graph with the vertex set $V$ and the edge set $E$.
Throughout this paper, we assume that graphs are simple and finite,
i.e., $|V| < \infty$ and $E \subset \{\{x,y\} \subset V \mid x \neq y\}$.
%For $x \in V$,
%the set of neighbors of $x$ is denoted by $N(x)$.
Define $\MC{A} = \MC{A}(\G)=\{ (x, y), (y, x) \mid \{x, y\} \in E \}$,
which is the set of \emph{symmetric arcs} of $\G$.
The origin $x$ and terminus $y$ of $a=(x, y) \in \MC{A}$ are denoted by $o(a)$ and $t(a)$, respectively.
We write the inverse arc of $a$ as $a^{-1}$.
We define several matrices on Grover walks.
Note that Grover walks are also referred to as arc-reversal walks or arc-reversal Grover walks,
and they are known as a special case of bipartite walks~\cite{chen2022hamiltonians}.
Let $\G = (V, E)$ be a graph, and set $\MC{A} = \MC{A}(\Gamma)$.
Denote by $I_V$ and $I_{\MC{A}}$ the identity matrices indexed by $V$ and $\MC{A}$, respectively.\MNOTE[green]{Change B.3 by yoshino. Reviewer 2 request to use $I_V$ and $I_\MC{A}$}
The \emph{boundary matrix} $d = d(\G) \in \MB{C}^{V \times \MC{A}}$ is defined by
\[ d_{x,a} = \frac{1}{\sqrt{\deg x}} \delta_{x, t(a)}, \MNOTE[green]{Change B.4}\]
where $\delta_{a,b}$ is the Kronecker delta.
%By directly calculating the entries based on the definition of matrix multiplication, it follows that
%\begin{equation} \label{1124-2}
%dd^* = I,
%\end{equation}
%where $I$ is the identity matrix.
The \emph{shift matrix} $R = R(\G) \in \MB{C}^{\MC{A} \times \MC{A}}$
is defined by $R_{a, b} = \delta_{a,b^{-1}}$.
%Clearly,
%\begin{equation}
Note that $R^2 = I_{\MC{A}}$\MNOTE[green]{Change B.3}.
%\end{equation}
Define the \emph{time evolution matrix} $U = U(\G) \in \MB{C}^{\MC{A} \times \MC{A}}$
by $U = R(2d^*d-I_{\MC{A}})$\MNOTE[green]{Change B.3}.
%The matrix $U$ is sometimes called the {\it Grover transfer matrix}.
The entries of $U$ are computed as
\begin{equation} \label{1223-1}
U_{a,b} = \frac{2}{\deg_{\G} t(b)} \delta_{o(a), t(b)} - \delta_{a,b^{-1}},
\end{equation}
whose substantial proof is in Lemma~5.1 of~\cite{kubota2021periodicity}.
The quantum walk defined by $U$ is called the \emph{Grover walk} on $\Gamma$.

%Note that
The matrix $U$ can be visually captured by regarding it as a linear mapping.
Let $\G = (V, E)$ be a graph.
We write the entries of a vector $\Psi \in \MB{C}^{\MC{A}}$
on the arcs of the graph as in Figure~\ref{48}.
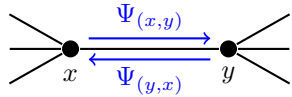
\begin{figure}[ht]
\begin{center}
\begin{tikzpicture}
[scale = 0.7,
line width = 0.8pt,
v/.style = {circle, fill = black, inner sep = 0.8mm},u/.style = {circle, fill = white, inner sep = 0.1mm}]
  \node[u] (1) at (-1.2, 0) {};
  \node[v] (2) at (0, 0) {};
  \node[v] (3) at (3, 0) {};
  \node[u] (4) at (-1.2, 0.68) {};
  \node[u] (5) at (-1.2, -0.68) {};
  \node[u] (7) at (4.2, 0) {};
  \node[u] (8) at (4.2, 0.68) {};
  \node[u] (9) at (4.2, -0.68) {};
  \node[u] (12) at (3.3, 0.2) {};
  \node[u] (13) at (5.7, 0.2) {};
  \draw (0,-0.5) node{$x$};
  \draw (3,-0.5) node{$y$};
  \draw (1) to (2);
  \draw[-] (2) to (4);
  \draw[-] (2) to (3);
  \draw[-] (5) to (2);
  \node[u] (10) at (0.3, 0.2) {};
  \node[u] (11) at (2.7, 0.2) {};
  \draw[draw= blue,->] (10) to (11);
  \node[u] (20) at (0.3, -0.2) {};
  \node[u] (21) at (2.7, -0.2) {};
  \draw[draw= blue,->] (21) to (20);
  \draw[-] (3) to (7);
  \draw[-] (3) to (8);
  \draw[-] (3) to (9);
  \draw (1.5,0.6) node[blue]{$\Psi_{(x,y)}$};
  \draw (1.5,-0.6) node[blue]{$\Psi_{(y,x)}$};
\end{tikzpicture}
\caption{The entries of a vector written on the arcs of the graph} \label{48}
\end{center}
\end{figure}
If an entry of $\Psi$ is $0$,
we omit the arc itself corresponding to the entry.
If an entry of $\Psi$ is $1$,
we may omit the value on the arc corresponding to the entry.
%\sout{This view verifies that the standard basis of %$\MB{C}^{\MC{A}}$ transitions by the action of $U$ as shown in %Figure~\ref{0325-1}.
%Roughly speaking,}
%This view is useful to capture how the standard basis of $\MB{C}^{\MC{A}}$ transitions by the action of $U$.
See Figure~\ref{0325-1}.
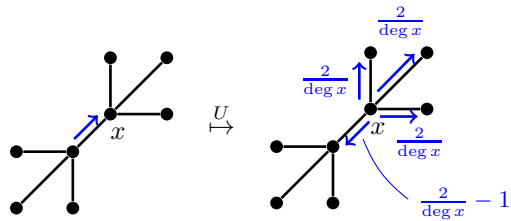
\begin{figure}[ht]
\begin{center}
\begin{tikzpicture}
[scale = 0.5,
v/.style = {circle, fill = black, inner sep = 0.6mm},
u/.style = {circle, fill = white, inner sep = 0.1mm}
]
\node[u] (x) at (0.7, 0) {$x$};
\node[u, white] (13) at (1.25, 2.7) {{\scriptsize$\frac{2}{\deg x}$}};
\node[u, white] (15) at (0, -2) {{\scriptsize$\frac{2}{\deg x} - 1$}};
\node[v] (1) at (0.5,0.5) {};
\node[v] (2) at (0.5, 2) {};
\node[v] (3) at (2, 2) {};
\node[v] (4) at (2, 0.5) {};
\draw[line width = 1pt] (1) to (2);
\draw[line width = 1pt] (1) to (3);
\draw[line width = 1pt] (1) to (4);
\node[v] (5) at (-0.5,-0.5) {};
\node[v] (6) at (-0.5, -2) {};
\node[v] (7) at (-2, -2) {};
\node[v] (8) at (-2, -0.5) {};
\draw[line width = 1pt] (1) to (5);
\draw[line width = 1pt] (5) to (6);
\draw[line width = 1pt] (5) to (7);
\draw[line width = 1pt] (5) to (8);
\node[u] (51L) at (-0.5, -0.2) {};
\node[u] (51R) at (0.2, 0.5) {};
\draw[draw = blue, line width = 1pt, ->] (51L) to (51R);
\end{tikzpicture}
\raisebox{45pt}{$\quad \overset{U}{\mapsto} \quad$}
\begin{tikzpicture}
[scale = 0.5,
v/.style = {circle, fill = black, inner sep = 0.6mm},
u/.style = {circle, fill = white, inner sep = 0.1mm}
]
\node[u] (15) at (3, -2) {\textcolor{blue}{{\small$\frac{2}{\deg x} - 1$}}};
\node[u] (12) at (-0.7, 1.25) {\textcolor{blue}{{\small$\frac{2}{\deg x}$}}};
\node[u] (13) at (1.3, 2.8) {\textcolor{blue}{{\small$\frac{2}{\deg x}$}}};
\node[u] (14) at (1.8, -0.4) {\textcolor{blue}{{\small$\frac{2}{\deg x}$}}};
\node[u] (x) at (0.7, 0) {$x$};
\node[v] (1) at (0.5,0.5) {};
\node[v] (2) at (0.5, 2) {};
\node[v] (3) at (2, 2) {};
\node[v] (4) at (2, 0.5) {};
\draw[line width = 1pt] (1) to (2);
\draw[line width = 1pt] (1) to (3);
\draw[line width = 1pt] (1) to (4);
\node[v] (5) at (-0.5,-0.5) {};
\node[v] (6) at (-0.5, -2) {};
\node[v] (7) at (-2, -2) {};
\node[v] (8) at (-2, -0.5) {};
\draw[line width = 1pt] (1) to (5);
\draw[line width = 1pt] (5) to (6);
\draw[line width = 1pt] (5) to (7);
\draw[line width = 1pt] (5) to (8);
\node[u] (15R) at (0.5, 0.2) {};
\node[u] (15L) at (-0.2, -0.5) {};
\draw[draw = blue, line width = 1pt, ->] (15R) to (15L);
\node[u] (12o) at (0.2, 0.7) {};
\node[u] (12t) at (0.2, 1.8) {};
\draw[draw = blue, line width = 1pt, ->] (12o) to (12t);
\node[u] (13o) at (0.65, 0.93) {};
\node[u] (13t) at (1.7, 2) {};
\draw[draw = blue, line width = 1pt, ->] (13o) to (13t);
\node[u] (14o) at (0.7, 0.3) {};
\node[u] (14t) at (1.8, 0.3) {};
\draw[draw = blue, line width = 1pt, ->] (14o) to (14t);
\draw[-, blue] (1.5,-1.9) to [bend left = 15] (0.3,-0.3);
\end{tikzpicture}
\caption{The action of $U$} \label{0325-1}
\end{center}
\end{figure}
An arrow transmits with weight $\frac{2}{\deg_{\G} x}$ except for the opposite arrow,
and reflects to the opposite arrow with weight $\frac{2}{\deg_{\G} x} -1$.
See Section~6 in \cite{kubota2021periodicity} 
\TR{as $\eta = 0$} 
\MNOTE{Change~A.1}
for details.
In particular,
an arrow is \TR{fully transmitted}\TB{perfectly transmitted}\MNOTE[green]{Change B.5}
if the degree of the terminus of the arc is $2$,
as shown in Figure~\ref{44}.
\begin{figure}[htb]
\begin{center}
\begin{tikzpicture}
[scale = 0.6,
line width = 0.8pt,
v/.style = {circle, fill = black, inner sep = 0.8mm},u/.style = {circle, fill = white, inner sep = 0.1mm}]
  \node[u] (1) at (-1.2, 0) {};
  \node[v] (2) at (0, 0) {};
  \node[v] (3) at (3, 0) {};
  \node[u] (4) at (-1.2, 0.68) {};
  \node[u] (5) at (-1.2, -0.68) {};
  \node[v] (6) at (6, 0) {};
  \node[u] (7) at (7.2, 0) {};
  \node[u] (8) at (7.2, 0.68) {};
  \node[u] (9) at (7.2, -0.68) {};
  \node[u] (10) at (0.3, 0.2) {};
  \node[u] (11) at (2.7, 0.2) {};
  \node[u] (12) at (3.3, 0.2) {};
  \node[u] (13) at (5.7, 0.2) {};
  \draw (1) to (2);
  \draw[-] (2) to (4);
  \draw[-] (2) to (3);
  \draw[-] (5) to (2);
  \draw[draw= blue,->] (10) to (11);
  \draw[-] (3) to (6);
  \draw[-] (6) to (7);
  \draw[-] (6) to (8);
  \draw[-] (6) to (9);
\end{tikzpicture}
\raisebox{3.5mm}{$\quad \overset{U}{\mapsto} \quad$}
\begin{tikzpicture}
[scale = 0.7,
line width = 0.8pt,
v/.style = {circle, fill = black, inner sep = 0.8mm},u/.style = {circle, fill = white, inner sep = 0.1mm}]
  \node[u] (1) at (-1.2, 0) {};
  \node[v] (2) at (0, 0) {};
  \node[v] (3) at (3, 0) {};
  \node[u] (4) at (-1.2, 0.68) {};
  \node[u] (5) at (-1.2, -0.68) {};
  \node[v] (6) at (6, 0) {};
  \node[u] (7) at (7.2, 0) {};
  \node[u] (8) at (7.2, 0.68) {};
  \node[u] (9) at (7.2, -0.68) {};
  \node[u] (10) at (0.3, 0.2) {};
  \node[u] (11) at (2.7, 0.2) {};
  \node[u] (12) at (3.3, 0.2) {};
  \node[u] (13) at (5.7, 0.2) {};
  \draw (4.5,0.7) node[blue]{};
  \draw (1) to (2);
  \draw[-] (2) to (4);
  \draw[-] (2) to (3);
  \draw[-] (5) to (2);
  \draw[draw= blue,->] (12) to (13);
  \draw[-] (3) to (6);
  \draw[-] (6) to (7);
  \draw[-] (6) to (8);
  \draw[-] (6) to (9);
\end{tikzpicture}
\caption{The action of $U$ in the case where the degree of the terminus of an arc is $2$} \label{44}
\end{center}
\end{figure}
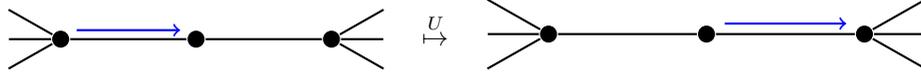

Let $\G$ be a graph,
and let $U = U(\G)$ be the time evolution matrix.
A vector $\Phi \in \MB{C}^{\MC{A}}$ is a \emph{state} if $\| \Phi \| = 1$.
For distinct states $\Phi$ and $\Psi$,
we say that \emph{perfect state transfer} occurs from $\Phi$ to $\Psi$ at time $\tau \in \MB{Z}_{\geq 1}$
if there exists $\gamma \in \MB{C}$ with norm one such that $U^{\tau}\Phi = \gamma \Psi$.
%As mentioned in Section~4 of \cite{chan2021pretty},
%the occurrence of perfect state transfer can be restated as follows.
%\begin{lem}[Section~4 in \cite{chan2021pretty}] \label{1201-2}
%Let $\G$ be a graph,
%and let $U = U(\G)$ be the time evolution matrix.
%Perfect state transfer occurs from a state $\Phi$ to a state $\Psi$ at time $\tau$
%if and only if $|\Span{ U^{\tau}\Phi, \Psi }| = 1$.
%\end{lem}
%Let $\G$ be a graph with the vertex set $V$.
A state $\Phi$ is said to be \emph{vertex type}
if there exists a vertex $x \in V$ such that $\Phi = d^* e_{x}$,
where $e_{x} \in \MB{C}^{V}$ is the unit vector defined by $(e_{x})_z = \delta_{x,z}$.
A vertex type state represents a state that has values only on arcs reaching toward some vertex.
Indeed,
letting $\TR{e_{a}}\TB{\e_a}\MNOTE[green]{Change B.7} \in \MB{C}^{\MC{A}}$ be the unit vector defined by $(\TR{e_{a}}\TB{\e_a}\MNOTE[green]{Change B.7})_z = \delta_{a,z}$,
we have
\[ d^* e_{x} = \frac{1}{\sqrt{\deg x}}\sum_{\substack{a \in \MC{A} \\ t(a) = x}} \TR{e_{a}}\TB{\e_a}\MNOTE[green]{Change B.7}. \]
This can be visually represented as shown in the left of Figure~\ref{1222-1}.
A vertex type state corresponds to the situation that
a quantum walker is only on some vertex
in the model where a walker moves on vertices.
See also the right of Figure~\ref{1222-1}.
We study perfect state transfer between vertex type states.
See~\cite{kubota2022perfect} for motivation to restrict states under consideration to vertex type states.

\begin{figure}[ht] 
\begin{center}
\begin{tikzpicture}
[scale = 0.7,
line width = 0.8pt,
v/.style = {circle, fill = black, inner sep = 0.8mm},
u/.style = {circle, fill = white, inner sep = 0.1mm}]
  \node[v] (1) at (2, 0) {};
  \node[v] (2) at (0, 2) {};
  \node[v] (3) at (-2, 0) {};
  \node[v] (4) at (0, 4.6) {};
  \node[u] (10) at (0.4, 2.4) {$x$};
  \node[u] (1012) at (1.6, 1.65) {\textcolor{blue}{$\frac{1}{\sqrt{3}}$}};
  \node[u] (1032) at (-1.7, 1.65) {\textcolor{blue}{$\frac{1}{\sqrt{3}}$}};
  \node[u] (1032) at (-0.6, 3.3) {\textcolor{blue}{$\frac{1}{\sqrt{3}}$}};
  \draw[-] (1) to (2);
  \draw[-] (2) to (3);
  \draw[-] (4) to (2);  
  \node[u] (11) at (2, 0.3) {};
  \node[u] (22) at (0.3, 2) {};
  \draw[draw= blue,->] (11) to (22);
  \node[u] (23) at (-0.3, 2) {};
  \node[u] (33) at (-2, 0.3) {};
  \draw[draw= blue,->] (33) to (23);
  \node[u] (a) at (-0.15, 4.4) {};
  \node[u] (b) at (-0.15, 2.2) {};
  \draw[draw= blue,->] (a) to (b);
  \node[u] (4L) at (-0.5, 5.1) {};
  \node[u] (4R) at (0.5, 5.1) {};
  \draw[-] (4) to (4L);
  \draw[-] (4) to (4R);
  \node[u] (3L) at (-2.65, 0) {};
  \node[u] (3R) at (-2, -0.65) {};
  \draw[-] (3) to (3L);
  \draw[-] (3) to (3R);
  \node[u] (1R) at (2.65, 0) {};
  \node[u] (1L) at (2, -0.65) {};
  \draw[-] (1) to (1L);
  \draw[-] (1) to (1R);
\end{tikzpicture}
\raisebox{50pt}{
$\qquad \leftrightarrow$
}
\raisebox{-9pt}{
\begin{tikzpicture}
[scale = 0.7,
baseline = -23pt,
line width = 0.8pt,
v/.style = {circle, fill = black, inner sep = 0.8mm},
u/.style = {circle, fill = white, inner sep = 0.1mm}]
  \node[u] (2v) at (1.2, 2.6) {{\footnotesize \textcolor{blue}{$\begin{bmatrix} 1/\sqrt{3} \\ 1/\sqrt{3} \\ 1/\sqrt{3} \end{bmatrix}$}}};
  \node[u] (1v) at (2.5, 1) {{\footnotesize $\begin{bmatrix} 0 \\ 0 \\ 0  \end{bmatrix}$}};
  \node[u] (3v) at (-2.5, 1) {{\footnotesize $\begin{bmatrix} 0 \\ 0 \\ 0  \end{bmatrix}$}};
  \node[u] (4v) at (1, 4.6) {{\footnotesize $\begin{bmatrix} 0 \\ 0 \\ 0  \end{bmatrix}$}};
%  \node[u] (10) at (0.4, 2.4) {$x$};
  \node[v] (1) at (2, 0) {};
  \node[v] (2) at (0, 2) {};
  \node[v] (3) at (-2, 0) {};
  \node[v] (4) at (0, 4.6) {};
%  \node[u, blue] (1032) at (-1.8, 1.65) {$\frac{1}{\sqrt{3}}$};
%  \node[u, blue] (1032) at (-1, 3.3) {$\frac{1}{\sqrt{3}}$};
  \draw[-] (1) to (2);
  \draw[-] (2) to (3);
  \draw[-] (4) to (2);  
%  \node[u] (11) at (2, 0.3) {};
%  \node[u] (22) at (0.3, 2) {};
%  \draw[draw= blue,->] (11) to (22);
%  \node[u] (23) at (-0.3, 2) {};
%  \node[u] (33) at (-2, 0.3) {};
%  \draw[draw= blue,->] (33) to (23);
%  \node[u] (a) at (-0.15, 4.4) {};
%  \node[u] (b) at (-0.15, 2.2) {};
%  \draw[draw= blue,->] (a) to (b);
  %
  \node[u] (4L) at (-0.5, 5.1) {};
  \node[u] (4R) at (0.5, 5.1) {};
  \draw[-] (4) to (4L);
  \draw[-] (4) to (4R);
  \node[u] (3L) at (-2.65, 0) {};
  \node[u] (3R) at (-2, -0.65) {};
  \draw[-] (3) to (3L);
  \draw[-] (3) to (3R);
  \node[u] (1R) at (2.65, 0) {};
  \node[u] (1L) at (2, -0.65) {};
  \draw[-] (1) to (1L);
  \draw[-] (1) to (1R);
\end{tikzpicture}
}
\caption{A vertex type state $d^* e_x$ with $\deg x = 3$} \label{1222-1}
\end{center}
\end{figure}
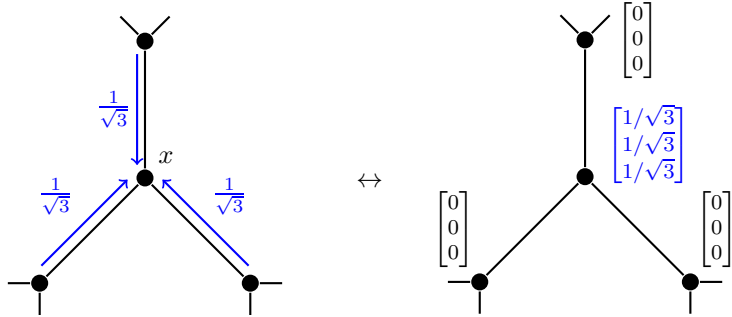

\begin{defi}
    For a graph $\Gamma$, we say that $\Gamma$ \emph{admits perfect state transfer} if perfect state transfer occurs for $U(\Gamma)$ between two distinct vertex type states.
\end{defi}

In this paper,
we completely characterize circulant graphs up to valency 4 that admit perfect state transfer.
Our main tools are symmetry of graphs and eigenvalues.
We will show a useful necessary condition for perfect state transfer to occur, which asserts that if it occurs between vertex type states $d^*e_x$ and $d^*e_y$,
then automorphisms that fix the vertex $x$ must also fix the vertex $y$.
See Corollary~\ref{1124-3} for details.
This fact is well-known in the context of continuous-time quantum walks.
As for eigenvalues, if perfect state transfer occurs,
then certain eigenvalues of the corresponding isotropic random walks must be the halves of algebraic integers.
Taking this into account,
we can use certain facts on the rings of integers of cyclotomic fields to completely characterize circulant graphs up to valency 4 that admit perfect state transfer.
The main theorems are the following two. 
For the symbol $X$ representing a circulant graph, refer to Section~\ref{sec:circulant}.

\begin{thm} \label{M0}
Let $X = X(\MB{Z}_{2l}, \{\pm a, l \} )$ be a connected $3$-regular circulant graph,
where $a \in \{1, \dots, l-1\}$.
Then, perfect state transfer does not occur on $X$ between vertex type states.
\end{thm}
%In particular, we provide the result for valency $4$ as the following first main theorem.
%Readers who want detailed terminologies and definitions can find them in later sections.

\begin{thm} \label{M1}
Let $X=X(\MB{Z}_{2l}, \{\pm a, \pm b\})$ be a connected $4$-regular circulant graph.
Let $x$, $y \in \MB{Z}_{2l}$.
Then, the graph $X$ admits perfect state transfer from $d^*e_x$ to $d^*e_{y}$ at minimum time $\tau  \in \MB{Z}_{\geq 1}$
if and only if $a+b=l$, $y=x+ l$ and one of the following hold.
\begin{enumerate}
\item $l$ is odd and $\tau=2l$.
\item $l \equiv 2 \pmod 4$ and $\tau=l$.
\end{enumerate}
\end{thm}

\TB{We supplement Theorem~\ref{M1}.
In Theorem 6.1 of \cite{zhan2019infinite},
Zhan showed that perfect state transfer occurs in the case~(i) of our Theorem~\ref{M1}.
We newly show that perfect state transfer also occurs in the case~(ii) and reveal that it does not occur in any other $4$-regular circulant graphs.
\MNOTE[green]{Change B.1}
}
\begin{comment}
The cases of valency $2$ and $3$ are provided in Theorems~\ref{0331-1} and~\ref{0331-2}, respectively.
Roughly speaking, if the valency is $2$, perfect state transfer occurs when the number of vertices is even,
and if the valency is $3$, it does not occur.

To study perfect state transfer on circulant graphs,
we investigate relationships between perfect state transfer and symmetry of graphs, where ``symmetry of graphs" precisely refers to ``automorphisms of graphs".
It has been observed in many examples that perfect state transfer occurs between two ``antipodal" vertices.
This paper makes the reason mathematically more clear.
Relationships between perfect state transfer and automorphisms of graphs are the following, which is our second main theorem.

\begin{thm}	\label{second}
Let $\aut$ be an automorphism of a graph $\G = (V, E)$,
and let $x,y \in V$.
If perfect state transfer occurs from $d^*e_x$ to $d^*e_y$ on $\G$ at time $\tau \in \MB{Z}_{\geq 1}$,
then the following hold.
\begin{enumerate}
\item Perfect state transfer occurs from $d^*e_{\aut(x)}$ to $d^*e_{\aut(y)}$ at time $\tau$.
\item $\Aut(\G)_x = \Aut(\G)_y$.
\end{enumerate}
\end{thm}
\end{comment}

\section{Preliminaries}
\begin{comment}
Let $[n] = \{0,1, \dots, n-1\}$ for a positive integer $n$.
See \cite{godsil2013} for basic terminologies related to graphs.
Let $\G =(V, E)$ be a graph with the vertex set $V$ and the edge set $E$.
Throughout this paper, we assume that graphs are simple and finite,
i.e., $|V| < \infty$ and $E \subset \{\{x,y\} \subset V \mid x \neq y\}$.
%For $x \in V$,
%the set of neighbors of $x$ is denoted by $N(x)$.
Define $\MC{A} = \MC{A}(\G)=\{ (x, y), (y, x) \mid \{x, y\} \in E \}$,
which is the set of the \emph{symmetric arcs} of $\G$.
The origin $x$ and terminus $y$ of $a=(x, y) \in \MC{A}$ are denoted by $o(a)$ and $t(a)$, respectively.
We write the inverse arc of $a$ as $a^{-1}$.

\subsection{Grover walks and related matrices}
We define several matrices on Grover walks.
Note that Grover walks are also referred to as arc-reversal walks or arc-reversal Grover walks,
and they are known as a special case of bipartite walks \cite{chen2022hamiltonians}.
Let $\G = (V, E)$ be a graph, and set $\MC{A} = \MC{A}(\Gamma)$.
The \emph{boundary matrix} $d = d(\G) \in \MB{C}^{V \times \MC{A}}$ is defined by
$d_{x,a} = \frac{1}{\sqrt{\deg x}} \delta_{x, t(a)}$,
where $\delta_{a,b}$ is the Kronecker delta.
\end{comment}
\TB{
This section provides supplementary details that were not covered in the introduction.
First, with respect to boundary matrices,
we have the following by direct calculation.}

\begin{equation} \label{1124-2}
dd^* = \TR{I}\TB{I_V.}\MNOTE[green]{Change B.3}
\end{equation}
\TR{where $I$ is the identity matrix.}
The \emph{discriminant} $P=P(\G) \in \MB{C}^{V \times V}$ is defined by $P = dRd^*$.
\TR{The matrix $P$ corresponds to the time evolution matrix of a certain random walk.}\MNOTE[green]{Change B.6}
\TB{The matrix $P$ is isomorphic to the transition matrix of the isotropic random walk.}
Indeed, the matrix $P$ is in some sense a normalized adjacency matrix of a graph. Note that the \emph{adjacency matrix} $A=A(\G) \in \MB{C}^{V \times V}$ of a graph $\G = (V,E)$ is defined by
\[ A_{x,y} = \begin{cases}
1 \qquad &\text{if $\{ x,y \} \in E$,} \\
0 \qquad &\text{otherwise.}
\end{cases} \]
If a graph is regular, then the two matrices are closely related:

\begin{lem}
Let $\G$ be a $k$-regular graph,
and let $A$ and $P$ be the adjacency matrix and the discriminant of $\G$, respectively.
Then we have $P = \frac{1}{k}A$.
%Therefore, the absolute values of eigenvalues of $P$ does not exceed $1$.
\end{lem}

This can be verified straightforwardly.
A more general claim and its proof can be found
in Theorem~3.1 of \cite{kubota2021quantum} and Proposition~3.3 of \cite{kubota2021quantum}.
Since the discriminant $P$ is a normal matrix,
it has a spectral decomposition.
Let $\lambda_1, \dots, \lambda_s$ be the distinct eigenvalues of $P$,
and let $W_i$ be the eigenspace associated to $\lambda_i$.
Let $d_i = \dim W_i$,
and let $u^{(i,1)}, \dots, u^{(i,d_i)}$ be an \TR{orthogonal }\TB{orthonormal}\MNOTE{Change A.2} basis of $W_i$.
Then, the matrix
\begin{equation} \label{1225-1}
E_i = \sum_{j=1}^{d_i} u^{(i,j)} (u^{(i,j)})^*
\end{equation}
is called the \emph{projection matrix} associated to the eigenvalue $\lambda_i$,
and it satisfies 
\begin{align}
P &= \sum_{i=1}^{s} \lambda_i E_i, \notag \\
E_i E_j &= \delta_{i,j} E_i \label{E1}, \\
E_i^* &= E_i, \label{E2} \\
\sum_{i=1}^s E_i &= I. \label{E3}
\end{align}
%\begin{equation} \label{0804-1}
%$P = \sum_{i=1}^{s} \lambda_i E_i$,
%\end{equation}
%and the projection matrices satisfy
%\begin{equation} \label{E1}
%E_i E_j = \delta_{i,j} E_i,
%\end{equation}
%\begin{equation} \label{E2}
%E_i^* = E_i,
%\end{equation}
%and
%\begin{equation} \label{E3}
%\sum_{i=1}^s E_i = I.
%\end{equation}
See \cite{meyer2000matrix} for the spectral decomposition of matrices.
Using the above properties, the following can be obtained.
%In this study, we use polynomials of $P$.
%The following can be obtained by using Equality~(\ref{0804-2}).

\begin{lem} \label{1203-1}
Let $M$ be a normal matrix with the spectral decomposition $M = \sum_{i=1}^{s} \lambda_i E_i$,
and let $p(x)$ be a polynomial.
Then we have $p(M) = \sum_{i=1}^{s} p(\lambda_i) E_i$.
\end{lem}

\section{Automorphisms of graphs} \label{s3}

In this section, we discuss relationships between symmetry of graphs and perfect state transfer.
For a wide readership,
we begin by introducing automorphisms of graphs and their associated permutation matrices.
Let $\G = (V, E)$ be a graph.
A mapping $\aut: V \to V$ is an \emph{automorphism} of $\G$
if $\aut$ is bijective, and $\{x,y\} \in E$
if and only if $\{\aut(x), \aut(y) \} \in E$.
We denote the set of all automorphisms of $\G$ by $\Aut(\G)$,
and let $\Aut(\G)_x = \{ \aut \in \Aut(\G) \mid \aut(x) = x \}$
for $x \in V$.
We note that automorphisms preserve the degrees of vertices,
i.e., 
\begin{equation} \label{1121-4}
\deg \aut(x) = \deg x
\end{equation}
for any $x \in V$ \TB{and $\aut \in \Aut(\Gamma)$}.\MNOTE{Change A.3}
See Lemma~1.3.1 in \cite{godsil2013} for its proof.

Let $\aut$ be an automorphism of a graph $\G$.
Since $\aut$ is bijective,
it defines the permutation matrix $M_{\aut} \in \MB{C}^{V \times V}$ given by $(M_{\aut})_{x,y} = \delta_{x, \aut(y)}$.
Basic properties of permutation matrices are
%\begin{lem}
%{\it
%With the above notation, we have
%\begin{enumerate}
%$M_{\aut}$ is an orthogonal matrix, i.e.,
\begin{equation} \label{1121-1}
M_{\aut}^\top = M_{\aut}^{-1} = M_{\aut^{-1}},
\end{equation}
and
\begin{equation} \label{1121-2}
M_{\aut}e_x = e_{\aut(x)}
\end{equation}
for any $x \in V$ \TB{and $\aut \in \Aut(\Gamma)$}.\MNOTE{Change A.3}
%\end{enumerate}
%}
%\end{lem}

%For the proof,
%it can be computed based on the definition of permutation matrices
%and hence we omit its details.
The action of automorphisms of a graph is naturally extended to the set of symmetric arcs.
Let $\aut$ be an automorphism of a graph $\G$,
and let $\MC{A} = \MC{A}(\G)$ be the symmetric arc set.
Define $\tilde{\aut}: \MC{A} \to \MC{A}$ by
$\tilde{\aut}(\left(x,y)\right) = (\aut(x), \aut(y))$.
Cleary, $\tilde{\aut}$ is a bijection.
Thus, the permutation matrix $N_{\tilde{\aut}} \in \MB{C}^{\MC{A} \times \MC{A}}$ is defined by $(N_{\tilde{\aut}})_{a,b} = \delta_{a, \tilde{\aut}(b)}$,
and it holds that 
\begin{equation} \label{proN}
N_{\tilde{\aut}}^{\top} = N_{\tilde{\aut}}^{-1} = N_{\tilde{\aut}^{-1}}
\end{equation}
as in Equality~(\ref{1121-1}).
Moreover, we have
\begin{equation} \label{1121-3}
N_{\tilde{\aut}}\TR{e_{a}}\TB{\e_a}\MNOTE[green]{Change B.7} = \TR{e_{\tilde{\aut}(a)}}\TB{\e_{\tilde{\aut}(a)}}\MNOTE[green]{Change B.7}
\end{equation}
for $a \in \MC{A}$ \TB{and $g \in \Aut(\Gamma)$}\MNOTE{Change~A.3}.
%Note that 
%\begin{equation} \label{1121-5}
%t(\tilde{\aut}(a)) = \aut(t(a))
%\end{equation}
%for $a \in \MC{A}$.
The following claims are relationships between permutation matrices defined by automorphisms and matrices associated with Grover walks,
which plays an important role in this study.

\begin{lem} \label{1122-3}
Let $\aut$ be an automorphism of a graph $\G$.
Then we have $d^* M_{\aut} = N_{\tilde{\aut}}d^*$.
\end{lem}

\begin{proof}
We want to show that $N_{\tilde{\aut}}^{-1} d^* M_{\aut} = d^*$.
For $a \in \MC{A}(\G)$ and $x \in V(\G)$,
we have
\begin{align*}
(N_{\tilde{\aut}}^{-1} d^* M_{\aut})_{a,x}
&= \TR{e_{a}}\TB{\e_a}\MNOTE[green]{Change B.7}^{\top} (N_{\tilde{\aut}}^{-1} d^* M_{\aut}) e_x \\
&= (N_{\tilde{\aut}} \TR{e_{a}}\TB{\e_a}\MNOTE[green]{Change B.7})^{\top} d^* (M_{\aut} e_x) \tag{by (\ref{proN})} \\
&= \TR{e_{\tilde{\aut}(a)}}\TB{\e_{\tilde{\aut}(a)}}\MNOTE[green]{Change B.7}^{\top} d^* e_{\aut(x)} \tag{by (\ref{1121-2}) and (\ref{1121-3})} \\
&= (d^*)_{\tilde{\aut}(a), \aut(x)} \\
&= d_{\aut(x), \tilde{\aut}(a)} \\
&= \frac{1}{\sqrt{\deg \aut(x)}} \delta_{\aut(x), t(\tilde{\aut}(a))} \\
&= \frac{1}{\sqrt{\deg x}} \delta_{\aut(x), \aut(t(a))} \tag{by (\ref{1121-4})} \\
&= \frac{1}{\sqrt{\deg x}} \delta_{x, t(a)} \\
&= d_{x,a} \\
&= (d^*)_{a, x}.
\end{align*}
Thus, we have $N_{\tilde{\aut}}^{-1} d^* M_{\aut} = d^*$, i.e., $d^* M_{\aut} = N_{\tilde{\aut}}d^*$.
\end{proof}

\begin{lem} \label{1122-1}
Let $\aut$ be an automorphism of a graph $\G$.
Then we have $U N_{\tilde{\aut}} = N_{\tilde{\aut}} U$.
\end{lem}

\begin{proof}
\TR{
We want to show that $N_{\tilde{\aut}}^{-1} U N_{\tilde{\aut}} = U$.
For $a, b \in \MC{A}(\G)$, we have
\begin{align*}
(N_{\tilde{\aut}}^{-1} U N_{\tilde{\aut}})_{a,b}
&= e_{a}^{\top} (N_{\tilde{\aut}}^{-1} U N_{\tilde{\aut}}) e_{b} \\
&= (N_{\tilde{\aut}} e_{a})^{\top} U (N_{\tilde{\aut}} e_{b}) \tag{by (\ref{proN})} \\
&= e_{\tilde{\aut}(a)}^{\top} U e_{\tilde{\aut}(b)} \tag{by (\ref{1121-3})} \\
&= U_{\tilde{\aut}(a), \tilde{\aut}(b)} \\
&= \frac{2}{\deg(t(\tilde{\aut}(b))} \delta_{o(\tilde{\aut}(a)), t(\tilde{\aut}(b))} - \delta_{\tilde{\aut}(a)^{-1}, \tilde{\aut}(b)} \tag{by (\ref{1223-1})} \\
&= \frac{2}{\deg(\aut(t(b))} \delta_{\aut(o(a)), \aut(t(b))} - \delta_{\tilde{\aut}(a^{-1}), \tilde{\aut}(b)} \\
&= \frac{2}{\deg(t(b))} \delta_{o(a), t(b)} - \delta_{a^{-1}, b} \tag{by (\ref{1121-4})} \\
&= U_{a,b}.
\end{align*}
Thus, we have $N_{\tilde{\aut}}^{-1} U N_{\tilde{\aut}} = U$,
i.e., $U N_{\tilde{\aut}} = N_{\tilde{\aut}} U$.
}
\TB{
Let $a, b \in \MC{A}(\Gamma)$.
In a similar way as in the proof of Lemma~\ref{1122-3}, we have
\begin{align*}
    (N_{\tilde{\aut}}^{-1} RN_{\tilde{\aut}})_ {a,b} 
%    &= \e_a^\top (N_{\tilde{\aut}}^{-1} RN_{\tilde{\aut}}) \e_b \\
%    &= (N_{\tilde{\aut}} \e_a)^\top R (N_{\tilde{\aut}} \e_b) \tag{by~\eqref{proN}} \\
%    &= \e_{\tilde{\aut}(a)}^\top R \e_{\tilde{\aut}(b)} \tag{by~\eqref{1121-3}}\\
    = R_{\tilde{\aut}(a), \tilde{\aut}(b)} 
    = \delta_{\tilde{\aut}(a), \tilde{\aut}(b)^{-1}} 
    = \delta_{a, b^{-1}}
    = R_{a,b}. 
\end{align*}
This means
$
    N_{\tilde {\aut}} R = R N_{\tilde{\aut}}.
$
Since $U = R(2d^*d - I_A)$, we have 
\begin{align*}
    N_{\tilde{\aut}}^* U N_{\tilde{\aut}} & =R  N_{\tilde{\aut}}^*\left(2 d^* d-I_A\right) N_{\tilde{\aut}} \\
    & =R\left(2 d^* M_\aut^* M_\aut d-I_A\right)    \tag{by Lemma~\ref{1122-3}} \\
    & =U.
\end{align*}
Thus, we have $U N_{\tilde{\aut}} = N_{\tilde{\aut}} U$.\MNOTE[green]{Change B.8}
}
\end{proof}

\TR{In the following,
we prove the second main theorem, Theorem~\ref{second}, provided in Section~\ref{s1}.
Indeed, the first and second claims in Theorem~\ref{second} are verified by Theorem~\ref{1124-1} and Corollary~\ref{1124-3}, respectively.
The first asserts, in brief,
that automorphisms preserve the occurrence of perfect state transfer between vertex type states.
}
%which can be explained colloquially as
%automorphisms preserve the occurrence of perfect state transfer between vertex type states.
\TB{These lemmas imply that automorphisms preserve the occurrence of perfect state transfer between vertex type states.}

\begin{thm} \label{1124-1}
Let $\aut$ be an automorphism of a graph $\G = (V, E)$,
and let $x,y \in V$.
If
\begin{equation} \label{1122-2}
U^{\tau}d^*e_x = \gamma d^*e_y
\end{equation}
for a positive integer $\tau$ and a complex number $\gamma$,
then we have $U^{\tau}d^*e_{\aut(x)} = \gamma d^*e_{\aut(y)}$.
In particular,
if perfect state transfer occurs from $d^*e_x$ to $d^*e_y$ on $\G$ at time $\tau$,
then so does from $d^*e_{\aut(x)}$ to $d^*e_{\aut(y)}$.
\end{thm}

\begin{proof}
Indeed,
\begin{align*}
U^{\tau}d^*e_{\aut(x)}
&= U^{\tau}d^* M_{\aut} e_{x} \tag{by (\ref{1121-2})} \\
&= U^{\tau} N_{\tilde{\aut}} d^* e_{x} \tag{by Lemma~\ref{1122-3}} \\
&= N_{\tilde{\aut}} U^{\tau} d^* e_{x} \tag{by Lemma~\ref{1122-1}} \\
&= N_{\tilde{\aut}} \gamma d^* e_{y} \tag{by (\ref{1122-2})} \\
&= \gamma N_{\tilde{\aut}} d^* e_{y} \\
&= \gamma d^* M_{\aut} e_{y} \tag{by Lemma~\ref{1122-3}} \\
&= \gamma d^* e_{\aut(y)} \tag{by (\ref{1121-2})}.
\end{align*}
\end{proof}

As a consequence of the above theorem,
a necessary condition for perfect state transfer to occur between vertex type states $d^*e_x$ and $d^*e_y$ can be formulated in terms of automorphism groups.
In short, an automorphism that fixes the vertex $x$ must also fix the vertex $y$.

\begin{cor} \label{1124-3}
Let $\aut$ be an automorphism of a graph $\G = (V, E)$,
and let $x,y \in V$.
If perfect state transfer occurs from $d^*e_x$ to $d^*e_y$ on $\G$,
then we have $\Aut(\G)_x = \Aut(\G)_y$.
\end{cor}

\begin{proof} %\sout{***}
%\sout{We note that although there are differences in symbols,
%Zhan shows that if perfect state transfer occurs from $d^*e_x$ to $d^*e_y$, then it also occurs from $d^*e_y$ to $d^*e_x$ in Corollary~5.4 of \cite{zhan2019infinite}.}
%\TR{(Should the notation distortion be mentioned?
%Alternatively, the reverse inclusion can be shown without %Zhan's result.)}
Let $\aut \in \Aut(\G)_x$,
and we want to show that $\aut(y) = y$.
Since $U^{\tau}d^*e_x = \gamma d^*e_y$ for a positive integer $\tau$ and a complex number $\gamma$, we have
\begin{align*}
d^* e_y &= \gamma^{-1} U^{\tau}d^*e_x \\
&= \gamma^{-1} U^{\tau}d^*e_{\aut(x)} \tag{by $\aut \in \Aut(\G)_x$} \\
&= \gamma^{-1} \gamma d^*e_{\aut(y)} \tag{by Theorem~\ref{1124-1}} \\
&= d^*e_{\aut(y)}.
\end{align*}
Multiply both sides by the matrix $d$.
Then by Equality~(\ref{1124-2}),
we have $\aut(y) = y$,
and hence $\aut \in \Aut(\G)_y$.
The reverse inclusion can be verified in a similar way.
\end{proof}
\TB{
In fact, the same statement holds true for continuous-time quantum walks as well.
See Lemma~4.1 in \cite{godsil2010can} for details.
}\MNOTE[black]{Change C.1}

\TB{
Using Corollary~\ref{1124-3}, we demonstrate in Example~\ref{ex:K3} how to verify that perfect state transfer does not occur between a specific pair of vertices. As this example shows, the corollary is useful if we know part of the automorphism group.\MNOTE[green]{Change B.9}
\begin{ex}  \label{ex:K3}
    Let $K_3$ be the complete graph of order $3$ with vertex set $[3]=\{0,1,2\}$ in Figure~\ref{fig:K3}.
    \begin{figure}[ht]
    \begin{center}
    \begin{tikzpicture} %1
    [scale = 0.6,
    every node/.style = {circle, fill = black, inner sep = 0.7mm}
    ]
    \node[label=right:$0$] (0) at (2,0) {};
    \node[label=left:$1$] (1) at (-1, 1.73) {};
    \node[label=left:$2$] (2) at (-1, -1.73) {};
    \draw[line width = 1pt] (0) -- (1) -- (2) -- (0);
    \end{tikzpicture}
    \end{center}
    \caption{The complete graph of order $3$ \label{fig:K3}}
\end{figure}
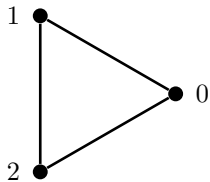
Let $g$ be the automorphism of $K_3$ defined by $g(0)=0$, $g(1)=2$ and $g(2)=1$.
Then, $g \in \Aut(K_3)_0$ and $g \not\in \Aut(K_3)_1$.
Hence, by Corollary~\ref{1124-3}, we see that perfect state transfer does not occur from $0$ to $1$.
\end{ex}
}
\TB{We will use Corollary~\ref{1124-3} to provide a necessary condition for perfect state transfer to occur. This assertion can also be interpreted as saying that the Grover walk captures symmetry of graphs through perfect state transfer. \MNOTE[green]{Change B.2. Slight changes from reviewer's suggestion}
}

\section{Circulant graphs}  \label{sec:circulant}
One of the main objects in this paper is circulant graphs,
which are Cayley graphs over the additive group $\MB{Z}_n = \MB{Z}/ n\MB{Z}$ of integers modulo $n$.
\TB{Here, we may assume that $n$ is a positive integer.}\MNOTE[black]{Change C.2}
If no confusion arises, we identify an integer $x \in \MB{Z}$ with the equivalent class $x + n\MB{Z} \in \MB{Z}_n$.

\TR{By using Corollary~\ref{1124-3},
we can determine two vertices in circulant graphs such that perfect state transfer occurs between those two vertices.}
\TB{
In this section, we present two useful lemmas for circulant graphs. 
Combining Corollary~\ref{1124-3} with some automorphisms of circulant graphs, we present Lemma~\ref{1214-3} to identify candidates for pairs of vertices in circulant graphs between which perfect state transfer can occur. 
After that, we introduce the concept of eigenvalue support to investigate whether perfect state transfer occurs. 
We then prove Lemma~\ref{1206-3}, which asserts that the eigenvalue support coincides with the set of eigenvalues for circulant graphs.}
\MNOTE{Change A.4: This paragraph is rewritten.}

We recall the definition and basic properties of circulant graphs.
Let $S$ be a subset of $\MB{Z}_n \setminus \{0\}$ such that
\begin{equation} \label{1206-1}
S = -S,
\end{equation}
i.e., $s \in S$ if and only if $-s \in S$.
The \emph{circulant graph} $X(\MB{Z}_n, S)$ is the graph with the vertex set $\MB{Z}_n$ and the edge set $\{ \{x,y\} \mid y-x \in S \}$.
In circulant graphs,
``rotations" and ``inversions" are automorphisms as claimed in the following lemma.

\begin{lem}[{pp.~8--9 \cite{godsil2013}}] \label{1212-1}
Let $X = X(\MB{Z}_n, S)$.
Then we have the following.
\begin{enumerate}
\item For $z \in \MB{Z}_n$, the mapping $\rho_z: \MB{Z}_n \to \MB{Z}_n$ defined by $\rho_z(x) = x+z$ is an automorphism of $X$.
\item The mapping $r: \MB{Z}_n \to \MB{Z}_n$ defined by
$r(x) = -x$ is an automorphism of $X$.
\end{enumerate}
\end{lem}

In considering perfect state transfer on circulant graphs between vertex type states, from Theorem~\ref{1124-1} we may assume that the initial state is $d^*e_0$ without loss of generality.
%\TB{Combining the above lemma and the discussion in Section~\ref{s3},}
Moreover, we obtain simple necessary conditions for perfect state transfer to occur on $X(\MB{Z}_n, S)$.

\begin{lem} \label{1214-3}
Let $X = X(\MB{Z}_n, S)$,
and let $x,y$ be distinct vertices of $X$.
If perfect state transfer occurs on $X$ from $d^*e_x$ to $d^*e_y$, then $n$ is even and $y = x+\frac{n}{2}$.
\end{lem}

\begin{proof}
Suppose that perfect state transfer occurs from $d^*e_x$ to $d^*e_y$.
From $\rho_{-x} \in \Aut(X)$ and Theorem~\ref{1124-1},
perfect state transfer occurs from $d^*e_{\rho_{-x}(x)} = d^*e_{0}$ to $d^*e_{\rho_{-x}(y)} = d^*e_{y-x}$.
Also,
the automorphism $r$ fixes the vertex $0$,
so Corollary~\ref{1124-3} implies that $r$ also fixes the vertex $y-x$,
i.e., $y-x = r(y-x) = -(y-x)$.
We have
\begin{equation} \label{1128-1}
2(y-x) = 0.
\end{equation}
We note that the operations and equal signs here are on $\MB{Z}_n$.
Since $x \neq y$ in $\MB{Z}_n$, we see that $n$ is even.
\TR{
Next, we prove $y = x + \frac{n}{2}$.
Let $n=2l$.
By Equality~(\ref{1128-1}),
we have $2(y-x) = 0$ on $\MB{Z}_{2l}$,
i.e., $y-x \equiv 0 \pmod l$.
Thus, there exists an integer $k$ such that $y-x = kl$.
In particular, we have $y-x \in \{0, l\}$ on $\MB{Z}_{2l}$.
Since $x \neq y$, we obtain $y-x = l$.}
\TB{Furthermore, we have either $y-x = 0$ or $y-x = \frac{n}{2}$.
As $x \neq y$, we conclude $y = x+\frac{n}{2}$.}\MNOTE{Change A.5}
\end{proof}

Next, we describe the eigenvalues and eigenvectors of circulant graphs.
A general discussion on eigenvalues of Cayley graphs on abelian groups is given in Section~1.4.9 in~\cite{brouwer2011spectra}.
In this paper,
we avoid representation theoretic discussion and explicitly provide eigenvectors of circulant graphs.
%Let $\chi$ be a character of the circulant graph $X=X(\mathbb{Z}_n,S)$, that is, a map $\chi : \mathbb{Z}_n \to \mathbb{C}^*$ satisfying $\chi(x+y)=\chi(x)\chi(y)$
%for $x,y \in \MB{Z}_n$.
%A character can be regarded as a vector in $\MB{C}^{\MB{Z}_n}$,
%which is known to be 
%%The vector $(\chi(x))_{x \in \mathbb{Z}_n}$ is
%an eigenvector of the adjacency matrix $A(X)$ belonging to the eigenvalue $\chi(S) := \sum_{s \in S}\chi(S)$
%(See Section~1.4.9 in~\cite{brouwer2011spectra}).
%Furthermore, it is well-known in representation theory that the set of characters of $\mathbb{Z}_n$ is an orthogonal basis of $\mathbb{C}^{\mathbb{Z}_n}$.
%In this paper, we explicitly provide the eigenvectors of $X$ for later discussion.
\TB{Let $[n] = \{0,1, \dots, n-1\}$ for a positive integer $n$.}\MNOTE[green]{Change B.11}
Let $\zeta_{n} = e^{\frac{2\pi}{n}i}$.
For $j \in [n]$, define the vector
$u_j \in \MB{C}^{\MB{Z}_n}$ by
\begin{equation} \label{1212-2}
(u_j)_k = \zeta_n^{jk}
\end{equation}
for $k \in [n]$.

%Then, $u_k$ is an eigenvector of $A(X)$ belonging to the eigenvalue $\sum_{s \in S} \zeta_n^{ks}$.
%Furthermore, the eigenvectors $u_0, u_1, \dots, u_{n-1}$ are pairwise orthogonal.

\begin{lem}[Section~1.4.9 in~\cite{brouwer2011spectra}] \label{lem:ortho}
Let $X = X(\MB{Z}_n, S)$.
%and let $\zeta_{n} = e^{\frac{2\pi}{n}i}$.
%For $k \in \{0,\ldots,n-1\}$, define the vector
%$\TR{\chi_{k}} \in \MB{C}^{\MB{Z}_n}$ by $(\chi_k)_j = \zeta_n^{jk}$.
Then $u_j$ is an eigenvector of $A(X)$ \TR{belonging to}\TB{of}\MNOTE[green]{Change B.10} the eigenvalue $\sum_{s \in S} \zeta_n^{js}$.
Furthermore, the eigenvectors $u_0, u_1, \dots, u_{n-1}$ are pairwise orthogonal.
\end{lem}
%\begin{proof}
%For each $i=0,\ldots,n-1$, the map $\chi : \mathbb{Z}_n \to \mathbb{C}^*$ defined as $\chi(j)=\zeta_n^{ij}$ is a character of $\mathbb{Z}_n$.
%Hence the desired result follows immediately.
%\end{proof}

%\TR{(Eigenvalue supports)}
\begin{defi}
Let $M \in \MB{C}^{n \times n}$ be a normal matrix
with the spectral decomposition $M = \sum_{i=1}^{s} \lambda_i E_i$,
where $\lambda_1, \dots, \lambda_s$ are the distinct eigenvalues of $M$,
and $E_1, \ldots,E_s$ are the orthogonal projections associated to these eigenvalues, respectively.
We denote by $\spec(M)$ the set of \TR{the }\MNOTE{Change A.6}distinct eigenvalues of $M$.
For a vector $x \in \MB{C}^n$, we define $\Theta_M(x) = \{ \lambda_i \in \spec(M) \mid E_i x \neq 0 \}$.
This is called the \emph{eigenvalue support} of the vector $x$ with respect to $M$.
\end{defi}
Note that $\spec(M)$ is simply a set, so the multiplicities of the eigenvalues are ignored.

\begin{lem} \label{1206-3}
Let $X = X(\MB{Z}_n, S)$,
and let $A$ be the adjacency matrix of $X$.
Then we have $\Theta_A(e_0) = \spec(A)$.
\end{lem}

\begin{proof}
It is enough to show that $\Theta_A(e_0) \supset \spec(A)$.
Let $\lambda \in \spec(A)$.
Define 
\[
I_\lambda = \left\{ j \in [n] \, \middle | \, \sum_{s \in S} \zeta_n^{js} = \lambda \right\}.
\]
Let $E_\lambda$ be the projection matrix associated to the eigenvalue $\lambda$.
%    We write
%    \begin{equation} \label{1212-2}
%        u_i := (\zeta_n^{ij})_{j=0}^{n-1}.
%    \end{equation}
Then by~(\ref{1212-2}) and~(\ref{1225-1}), we have %Lemma~\ref{lem:ortho},
    \begin{align*}
        E_\lambda = \frac{1}{n} \sum_{j \in I_\lambda} u_j u_j^*. 
        %\sum_{j \in I_\lambda} \frac{u_j u_j^*}{\langle u_j, u_j \rangle}
    \end{align*}
    Hence
    \begin{align*}
        E_\lambda e_0
        = \frac{1}{n} \left( \sum_{j \in I_\lambda} u_j u_j^* \right) e_0
        = \frac{1}{n} \sum_{j \in I_\lambda} u_j.
    \end{align*}
    By Lemma~\ref{lem:ortho},
    $u_i$'s are pairwise orthogonal,
    so they are linearly independent.
    This implies that $E_\lambda e_0 \neq 0$.
    Therefore, we have $\lambda \in \Theta_A(e_0)$, i.e.,
    $\spec(A) = \Theta_A(e_0)$.
\end{proof}

\section{Cycle graphs}
As a simple application of the previous sections,
we characterize circulant graphs with valency $2$ that admit perfect state transfer.
Connected $2$-regular graphs are cycle graphs.
Thus, it is sufficient to consider only $C_n = X(\MB{Z}_n, \{\pm 1\})$ without loss of generality.
By Lemma~\ref{1214-3},
if the cycle graph $C_n$ admits perfect state transfer,
then $n$ is even.
Moreover, let $n = 2l$,
and if perfect state transfer occurs from $d^* e_x$,
then it must occur from $d^* e_x$ to $d^* e_{x+l}$.
Conversely, it can be easily verified that
$C_{2l}$ indeed admits perfect state transfer
by visually capturing the action of the time evolution matrix $U$.
Readers should recall the action of $U$ described in Figure~\ref{44}.
For simplicity, we confirm this with $C_6$
as shown in Figure~\ref{0327-1},
but readers will be convinced that
the same is true for general $C_{2l}$.

\begin{figure}[hbt]
\begin{center}
\begin{tikzpicture} %1
[scale = 0.6,
v/.style = {circle, fill = black, inner sep = 0.7mm},
u/.style = {circle, fill = white, inner sep = 0.1mm}
]
\node[v] (1) at (2,0) {};
\node[v] (2) at (1, 1.73) {};
\node[v] (3) at (-1, 1.73) {};
\node[v] (4) at (-2, 0) {};
\node[v] (5) at (-1, -1.73) {};
\node[v] (6) at (1, -1.73) {};
\draw[line width = 1pt] (1) to (2);
\draw[line width = 1pt] (2) to (3);
\draw[line width = 1pt] (3) to (4);
\draw[line width = 1pt] (4) to (5);
\draw[line width = 1pt] (5) to (6);
\draw[line width = 1pt] (6) to (1);
\draw[draw = blue, line width = 1pt, ->] (-1.3, 1.75) to node[left] {$\textcolor{blue}{\frac{1}{\sqrt{2}}}$} (-2.16, 0.3);
\draw[draw = blue, line width = 1pt, ->] (-1.3, -1.75) to node[left] {$\textcolor{blue}{\frac{1}{\sqrt{2}}}$} (-2.16, -0.3);
\end{tikzpicture}
\raisebox{28pt}{$\quad \overset{U}{\mapsto} \quad$}
\begin{tikzpicture} %2
[scale = 0.6,
v/.style = {circle, fill = black, inner sep = 0.7mm},
u/.style = {inner sep = 0.1mm}%{circle, fill = transparent!50, inner sep = 0.1mm}
]
%\node[u, blue] (32) at (0, 0.7) {{\scriptsize$\frac{2}{2} - 1$}};
%\node[u, blue] (21) at (2, 1.5) {{\scriptsize$\frac{2}{2}$}};
\node[v] (1) at (2,0) {};
\node[v] (2) at (1, 1.73) {};
\node[v] (3) at (-1, 1.73) {};
\node[v] (4) at (-2, 0) {};
\node[v] (5) at (-1, -1.73) {};
\node[v] (6) at (1, -1.73) {};
\draw[line width = 1pt] (1) to (2);
\draw[line width = 1pt] (2) to (3);
\draw[line width = 1pt] (3) to (4);
\draw[line width = 1pt] (4) to (5);
\draw[line width = 1pt] (5) to (6);
\draw[line width = 1pt] (6) to (1);
\node[u] (23o) at (0.7,1.45) {};
\node[u] (23t) at (-0.7, 1.45) {};
\node[u] (21o) at (1.3,1.7) {};
\node[u] (21t) at (2.15, 0.3) {};
\draw[draw = blue, line width = 1pt, ->] (-2.16, 0.3) to node[left] {$\textcolor{blue}{\frac{1}{\sqrt{2}}}$} (-1.3, 1.75);
\draw[draw = blue, line width = 1pt, ->] (-2.16, -0.3) to node[left] {$\textcolor{blue}{\frac{1}{\sqrt{2}}}$} (-1.3, -1.75);
\end{tikzpicture}
\raisebox{28pt}{$\quad \overset{U}{\mapsto} \quad$}
\begin{tikzpicture} %3
[scale = 0.6,
v/.style = {circle, fill = black, inner sep = 0.7mm},
u/.style = {circle, fill = white, inner sep = 0.1mm},
baseline = -1.12cm
]
\node[v] (1) at (2,0) {};
\node[v] (2) at (1, 1.73) {};
\node[v] (3) at (-1, 1.73) {};
\node[v] (4) at (-2, 0) {};
\node[v] (5) at (-1, -1.73) {};
\node[v] (6) at (1, -1.73) {};
\draw[line width = 1pt] (1) to (2);
\draw[line width = 1pt] (2) to (3);
\draw[line width = 1pt] (3) to (4);
\draw[line width = 1pt] (4) to (5);
\draw[line width = 1pt] (5) to (6);
\draw[line width = 1pt] (6) to (1);
\draw[draw = blue, line width = 1pt, ->] (-0.7, 2) to node[above] {$\textcolor{blue}{\frac{1}{\sqrt{2}}}$} (0.7, 2);
\draw[draw = blue, line width = 1pt, ->] (-0.7, -2) to node[below] {$\textcolor{blue}{\frac{1}{\sqrt{2}}}$} (0.7, -2);
\end{tikzpicture}
\raisebox{28pt}{$\quad \overset{U}{\mapsto} \quad$}
\begin{tikzpicture} %4
[scale = 0.6,
v/.style = {circle, fill = black, inner sep = 0.7mm},
u/.style = {circle, fill = white, inner sep = 0.1mm}
]
\node[v] (1) at (2,0) {};
\node[v] (2) at (1, 1.73) {};
\node[v] (3) at (-1, 1.73) {};
\node[v] (4) at (-2, 0) {};
\node[v] (5) at (-1, -1.73) {};
\node[v] (6) at (1, -1.73) {};
\draw[line width = 1pt] (1) to (2);
\draw[line width = 1pt] (2) to (3);
\draw[line width = 1pt] (3) to (4);
\draw[line width = 1pt] (4) to (5);
\draw[line width = 1pt] (5) to (6);
\draw[line width = 1pt] (6) to (1);
\draw[draw = blue, line width = 1pt, ->] (1.3, 1.75) to node[right] {$\textcolor{blue}{\frac{1}{\sqrt{2}}}$} (2.16, 0.3);
\draw[draw = blue, line width = 1pt, ->] (1.3, -1.75) to node[right] {$\textcolor{blue}{\frac{1}{\sqrt{2}}}$} (2.16, -0.3);
\end{tikzpicture}
\caption{Visual proof that $C_6$ admits perfect state transfer at time $3$\MNOTE{Change A.7}} \label{0327-1}
\end{center}
\end{figure}
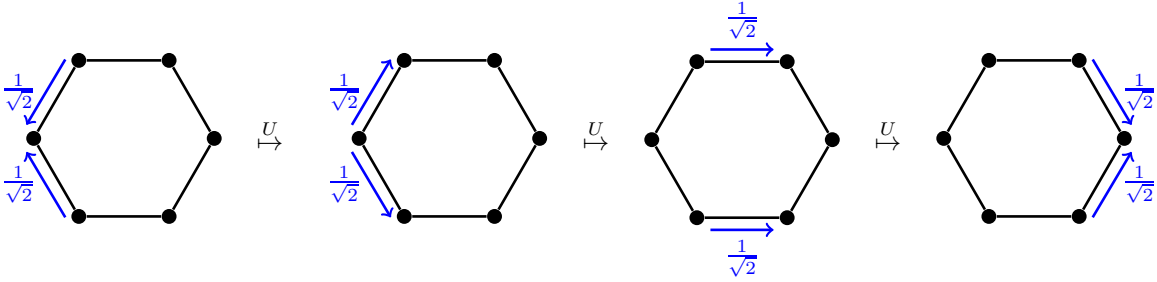

In summarizing the above discussion,
we have the following.

\begin{thm} \label{0331-1}
Let $C_n$ be the cycle graph on $n$ vertices.
\begin{enumerate}
\item If $n$ is odd, then $C_n$ does not admit perfect state transfer between distinct vertex type states.
\item If $n$ is even, say $n = 2l$,
then $C_n$ admits perfect state transfer from $d^*e_x$ to $d^*e_{x+l}$ at minimum time $l$.
\end{enumerate}
\end{thm}

\section{Characterization of perfect state transfer}

For circulant graphs with valency 3 or greater,
the superposition of quantum walkers,
which corresponds to the arrows in our model,
is too complex to consider in the same way
as in the previous section.
Therefore, it is necessary to switch strategies to use of eigenvalues.
The main goal of this section is to characterize perfect state transfer by eigenvalues of graphs.
We slightly generalize Zhan's result and connect it to Chebyshev polynomials.

The \emph{Chebyshev polynomial of the first kind}, denoted by $T_n(x)$,
is the polynomial defined by $T_0(x) = 1$, $T_1(x) = x$ and
$T_{n}(x) = 2xT_{n-1}(x) - T_{n-2}(x)$ for $n \geq 2$.
It is well-known that 
%\begin{equation}
$T_n(\cos \theta) = \cos (n \theta)$,
and hence we have the following.
%\end{equation}
%This implies
%\begin{equation}
%|T_n(x)| \leq 1
%\end{equation}
%for $|x| \leq 1$.

\begin{lem} \label{1203-5}
Let $\lambda \in [-1,1]$.
\begin{enumerate}
\item $|T_\tau(\lambda)| \leq 1$.
\item $T_{\tau}(\lambda) = 1$ if and only if $\lambda = \cos \frac{j}{\tau} \pi$ for some even $j$.
\item $T_{\tau}(\lambda) = -1$ if and only if $\lambda = \cos \frac{j}{\tau} \pi$ for some odd $j$.
\end{enumerate}
\end{lem}

\begin{lem}	\label{lem:r(P)}
Let $P$ be the discriminant of a graph $\Gamma$, and $\tau$ be a non-negative integer.
Then, the spectral radius of $P$ is at most $1$, and that of $T_\tau(P)$ is at most $1$.
In particular, $|T_\tau(\lambda)| \leq 1$ for every $\lambda \in \spec(P)$.
\end{lem}
\begin{proof}
    It follows from~\cite[Proposition~3.3]{kubota2021quantum} with $\eta = 0$ that the spectral radius of $P$ is at most $1$.
    This together with Lemma~\ref{1203-5} implies that the spectral radius of $T_\tau(P)$ is at most $1$.
\end{proof}

Furthermore, Chebyshev polynomials connect Grover walks and random walks.
Kubota and Segawa presented the following equality
and obtained a necessary condition on eigenvalues for perfect state transfer to occur~\cite{kubota2022perfect}.

\begin{lem}[Lemma~3.1 in \cite{kubota2022perfect}] \label{1201-1}
Let $\G$ be a graph with the time evolution matrix $U$ and the discriminant $P$.
Then we have $d U^{\tau} d^{*} = T_{\tau}(P)$ for \TR{a}\TB{any}\MNOTE[green]{Change~B.13} non-negative integer $\tau$.
\end{lem}

%\begin{thm}[Theorem~3.3 in \cite{kubota2022perfect}] \label{1203-2}
%Let $\G$ be a graph with the discriminant $P$,
%and let $x,y$ be vertices of $\G$.
%If perfect state transfer occurs from $d^*e_x$ to $d^*e_y$ at time $\tau$,
%then $T_{\tau}(\lambda) = \pm 1$ holds for any $\lambda \in \Theta_{P}(e_x)$.
%\end{thm}

Interestingly, an equality analogous to Lemma~\ref{1201-1} has recently been presented also in another quantum walk model, which is called vertex-face walks.
For more details see Section~4 of the paper written by Guo and Schmeits~\cite{guo2022perfect}.

Zhan has characterized in Theorem~5.3 of~\cite{zhan2019infinite} the occurrence of perfect state transfer in terms of eigenvalues and eigenprojections.
We slightly extend Zhan's theorem by using Lemma~\ref{1201-1}.
%The following is immediately derived from properties of the Chebyshev polynomials.
%Only the statements are provided for later discussion.
%%%
%The following is shown by Zhan~\cite{zhan2019infinite}.
%The same claim is cited,
%but the setting is adapted to ours.
%
%\begin{lem}[Lemma~5.1 in \cite{zhan2019infinite}] \label{1205-1}
%Let $\G = (V,E)$ be a regular graph,
%and let $x,y \in V$.
%We assume that a state $\Psi \in \MB{C}^{\MC{A}}$ satisfies $\Psi_a = 0$ if $t(a) \neq y$.
%If $U^{\tau} d^* e_x = \Psi$ for some positive integer $\tau$, then we have $\Psi = d^*e_y$.
%\end{lem}
%
%\sout{The following theorem is a characterization of perfect state transfer associated with Chebyshev polynomials.}
Our claim does not require that graphs are regular.
%\TB{We present two lemmas before stating our theorem.}

\TB{       
\begin{lem}[Section~4 in \cite{chan2021pretty}] \label{1201-2} \MNOTE[black]{Change C.3}
Let $\G$ be a graph,
and let $U = U(\G)$ be the time evolution matrix.
Perfect state transfer occurs from a state $\Phi$ to a state $\Psi$ at time $\tau$
if and only if $|\Span{ U^{\tau}\Phi, \Psi }| = 1$.
\end{lem}
}

\TR{
\begin{lem} \label{lem:CS}
	For two vectors $X$ and $Y$ with $\|X\| \leq 1$ and $\|Y\|=1$,
	$\langle X,Y \rangle = 1$ if and only if $X = Y$.
\end{lem}
}
\TR{
\begin{proof}
    We assume $\langle X,Y \rangle = 1$.
    We have $1 = |\langle X,Y \rangle | \leq \|X\|\|Y\| =  \|X\| \leq 1$ by the Cauchy-Schwarz inequality.
    Hence, both equalities hold, i.e, $\|X\|=1$ and there exists a complex number $k$ such that $X=kY$.
    Therefore, $1=\langle X,Y\rangle = k\langle Y,Y \rangle=k$.
    The converse is clear.
    \MNOTE[black]{Change B.15}
\end{proof}
}

\begin{thm} \label{pst}
Let $\G = (V, E)$ be a graph.
Let $x,y \in V$ and $\tau \in \mathbb{Z}_{\geq 1}$.
For each eigenvalue $\lambda$ of $P=P(\G)$,
let $E_{\lambda}$ be the eigenprojection of $P$ associated to $\lambda$.
The following are equivalent.
\begin{enumerate}[\textup{(}A\textup{)}]
	\item Perfect state transfer occurs on $\G$ from $d^*e_x$ to $d^*e_y$ at time $\tau$.\label{pst:A}
	\item %$||T_{\tau}(P)e_x|| = 1$ and 
		There exists $\gamma \in \{\pm 1\}$ such that $T_{\tau}(P)e_x = \gamma e_y$.\label{pst:B}
	\item There exists $\gamma \in \{\pm 1\}$ such that all of the following three conditions are satisfied:\label{pst:C}
	\begin{enumerate}[\textup{(}a\textup{)}]
		\item For any $\lambda \in \spec(P)$, $E_{\lambda}e_x = \pm E_{\lambda}e_y$.\label{pst:C:a}
		\item For any $\lambda \in \Theta_{P}(e_x)$,\label{pst:C:b}
			if $E_{\lambda}e_x = \gamma E_{\lambda}e_y$,
			then $\lambda = \cos \frac{j}{\tau}\pi$ for some even $j$.
		\item For any $\lambda \in \Theta_{P}(e_x)$,\label{pst:C:c}
			if $E_{\lambda}e_x = -\gamma E_{\lambda}e_y$,
			then $\lambda = \cos \frac{j}{\tau}\pi$ for some odd $j$.
	\end{enumerate}
\end{enumerate}
Moreover,
if $\G$ is regular, then $\gamma = 1$ in both~\ref{pst:B} and~\ref{pst:C}.
\end{thm}

\begin{proof}
    Let $U$ be the time evolution matrix of $\Gamma$,
    and $P = \sum_{\lambda \in \spec(P)} \lambda E_\lambda$ the spectral decomposition.
 
    We assume~\ref{pst:A}, and prove~\ref{pst:B}.
    \TR{We set $\gamma := (d^*e_y)^* U^\tau (d^* e_x)$, and verify $\gamma \in \{\pm 1\}$ and $T_{\tau}(P)e_x = \gamma e_y$.
    Indeed,  we have $|\gamma| = 1$ from Lemma~\ref{1201-2}.
    Since $d$ and $U$ are real, we obtain $\gamma \in \{\pm 1\}$.
    In addition, we have $\| T_\tau(P) e_x\| \leq 1$ by Lemma~\ref{lem:r(P)}.
    Also, we obtain $\gamma = e_y^* T_\tau(P) e_x$ by Lemma~\ref{1201-1}.
    We can rewrite this as $1 = \langle T_\tau(P)e_x, \gamma e_y \rangle$.
    Hence by Lemma~\ref{lem:CS}, we have $T_\tau(P)e_x = \gamma e_y$ as desired.}
    \TB{By~\ref{pst:A}, we have $U^\tau d^* e_x = \gamma d^* e_y$ for some $\gamma \in \mathbb{C}$ with $|\gamma| = 1$. Since $U = R(2d^* d - \TR{I}\TB{I_{\MC{A}}}\MNOTE[green]{Change B.3})$ and $d^*$ are real matrices, $\gamma \in \{\pm 1\}$ follows.
    By~\ref{pst:A} again and Lemma~\ref{1201-1}, we have\MNOTE{Change A.8}
    \[
    T_\tau(P) e_x = (dU^\tau d^*) e_x = d (\gamma d^* e_y) = (dd^*) \gamma e_y = \gamma e_y.\MNOTE[green]{Chane B.15 (=A.8)}
    \]
    }
	
    Next, we assume~\ref{pst:B}, and prove~\ref{pst:C}.
    Take $\gamma \in \{\pm 1\}$ such that $T_\tau(P)e_x = \gamma e_y$.
    Since
    \begin{align*}
	1 
	= \| \gamma e_y \|^2 
	= \| T_\tau(P) e_x \|^2
	= \| \sum_{\lambda \in \spec(P)} T_\tau(\lambda) (E_\lambda e_x) \|^2
        = \sum_{\lambda \in \spec(P)} |T_\tau(\lambda)|^2 \|E_\lambda e_x \|^2
        \leq \sum_{\lambda \in \spec(P)} \|E_\lambda e_x \|^2=1
    \end{align*}
    by Lemmas~\ref{1203-1} and~\ref{lem:r(P)},
    we have 
    $
	T_\tau(\lambda) \in \{\pm 1\}$ for every $\lambda \in \Theta_P(e_x).
    $
    Also, from $T_\tau(P) e_x = \gamma e_y$, we have by Lemma~\ref{1203-1},
    $T_\tau(\lambda) E_\lambda e_x = \gamma E_\lambda e_y$ for every $\lambda \in \spec(P)$.
    These together with Lemma~\ref{1203-5} implies~\ref{pst:C}.

    We assume~\ref{pst:C}, and prove~\ref{pst:A}.
    By Lemmas~\ref{1201-1} and~\ref{1203-1}, we have
    \begin{align}	\label{pst:1}
	(d^*e_y)^* U^\tau (d^* e_x)
	= e_y^* T_\tau(P) e_x 
	= \sum_{\lambda \in \spec(P)} T_\tau(\lambda) (E_\lambda e_y)^* (E_\lambda e_x).
    \end{align}
    Let $\Theta_\pm := \{ \lambda \in \Theta_P(e_x) \mid E_\lambda e_x = \pm \gamma E_\lambda e_y\}$.
    Then, we have
    \begin{align*}
       \sum_{\lambda \in \spec(P)} T_\tau(\lambda) (E_\lambda e_y)^* (E_\lambda e_x)
        &=\sum_{\lambda \in \Theta_P(e_x)} T_\tau(\lambda) (E_\lambda e_y)^* (E_\lambda e_x) \\
        &= \sum_{\lambda \in \Theta_+} T_\tau(\lambda) (E_\lambda e_y)^* (E_\lambda e_x)
        +\sum_{\lambda \in \Theta_-} T_\tau(\lambda) (E_\lambda e_y)^* (E_\lambda e_x) \\
        &= \sum_{\lambda \in \Theta_+}  (E_\lambda e_y)^* (E_\lambda e_x)
        +\sum_{\lambda \in \Theta_-} - (E_\lambda e_y)^* (E_\lambda e_x) \tag{by~\ref{pst:C:b} and~\ref{pst:C:c}} \\
        &= \sum_{\lambda \in \Theta_+}  \gamma (E_\lambda e_y)^* (E_\lambda e_y)
        +\sum_{\lambda \in \Theta_-} \gamma (E_\lambda e_y)^* (E_\lambda e_y) \\
        &= \gamma \sum_{\lambda \in \spec(P)} (E_\lambda e_y)^* (E_\lambda e_y) \\
        &=\gamma.
    \end{align*}
    By Lemma~\ref{1201-2}, we derive~\ref{pst:A} as desired.

    \TB{
    Finally, we assume that $\Gamma$ is regular of valency $k$ and~\ref{pst:A} holds, and prove $\gamma = 1$ in~\ref{pst:B} and~\ref{pst:C}.
    By~\ref{pst:A}, we have $U^\tau d^* e_x = \gamma d^* e_y$ for some $\gamma \in \{\pm 1\}$.
    As discussed above, \ref{pst:B} and~\ref{pst:C} hold for this $\gamma$.
    By the assumption that $\Gamma$ is regular, we see that $P = A/k$,     
    where $A$ is the adjacency matrix of $\Gamma$.
    Hence, $1$ is an eigenvector of $P$, and the all-ones vector $\mathbf{j}$ is an eigenvector of $P$ belonging to $1$.
    This implies $1 \in \Theta_P(e_x)$.
    By~\ref{pst:C:b} and~\ref{pst:C:c} in \ref{pst:C}, $E_1 e_x = \gamma E_1 e_y$ holds.
    Thus, 
    \begin{align*}
        1 = \langle \mathbf{j} , e_x \rangle
        = \langle \mathbf{j} , E_1 e_x \rangle
        = \langle \mathbf{j} , \gamma E_1 e_y \rangle
        = \gamma \langle \mathbf{j} , e_y \rangle 
        = \gamma.
    \end{align*}
    Therefore, $\gamma = 1$ holds.}\MNOTE{Change A.9}
\end{proof}

Currently, we have not found examples of non-regular graphs that admit perfect state transfer with $\gamma = -1$.

\section{Circulant graphs with valency $3$}

In this section,
we show that $3$-regular circulant graphs do not admit perfect state transfer.
We note from~(\ref{1206-1}) that any 3-regular circulant graph $X(\MB{Z}_n, S)$ requires that $n$ is even, $a \neq 0, \frac{n}{2}$ and $S = \{\pm a, \frac{n}{2}\}$.

\TR{We remark some notations.}\MNOTE{Change A.10}
Let $a',b' \in \MB{Z}$, and set $a := a' + n\MB{Z}$ and $b := b' + n\MB{Z}$.
In the case of $n \in 2\MB{Z}$, $a$ is said to be \emph{odd} if $a'$ is an odd integer, and \emph{even} otherwise.
\TR{Two }\TB{The}\MNOTE{Change A.11 } elements $a$ and $b$ are said to be \emph{coprime} if the greatest common divisor of $a'$, $b'$ and $n$ equals $1$.
In other words, they are coprime if $n_1 a + n_2 b = 1$ in $\MB{Z}_n$ for some integers $n_1$ and $n_2$.
Also, for a function $f$, we will write $f(a')$ as $f(a)$ if $f(a')=f(a'')$ for any $a'' \in a$.
For example, $e^{\frac{2\pi}{n}a} = e^{\frac{2\pi}{n}a'}$ and $\cos(\frac{2\pi a}{n}) = \cos(\frac{2\pi a'}{n})$.

Let $X = X(\MB{Z}_{2l}, \{\pm a, l \} )$ be a $3$-regular circulant graph,
and let $P = P(X)$.
Then by Lemma~\ref{lem:ortho}, the eigenvalues of $P$ are
\[
\lambda_j = \frac{\zeta_{2l}^{aj} + \zeta_{2l}^{-aj} + \zeta_{2l}^{jl}}{3}
= \frac{2\cos \frac{aj}{l}\pi + (-1)^j}{3}
\]
for $j \in [2l]$.
Note that for $X$ to be connected, $a$ and $l$ are coprime since there exists a path from $0$ to $1$, i.e., there exist \TB{non-negative} integers $n_1$ and $n_2$ such that $n_1 a + n_2 l = 1$ in $\MB{Z}_{2l}$.
\TB{Here, the length of the path is $n_1+n_2$.}\MNOTE[green]{Change B.17}

\begin{defi}
	Let $r$ be a rational number $p/q$ with $p \in \MB{Z}$ and $q \in \MB{Z}_{\geq 1}$ where $p$ and $q$ are coprime.
	Then define $N(r) := q$.
\end{defi}

\begin{lem}[Fact~1.1 in \cite{berger2018linear}] \label{0115-2}
Let $r \in \mathbb{Q}$.
The numbers $1$ and $\cos (r \pi)$ are $\mathbb{Q}$-linearly independent if and only if $N(r) \geq 4$.
In particular, $\MB{Q} \cap \{ \cos(r \pi) \mid r \in \MB{Q} \} = \{0,\pm 1 ,\pm \frac{1}{2}\}.$
\end{lem}

\begin{lem}[Theorem~1.2 in \cite{berger2018linear}] \label{1208-1}
Let $r_1, r_2$ be rational numbers such that neither $r_1-r_2$ nor $r_1+r_2$ is an integer. 
Then the following are equivalent.
\begin{enumerate}
\item The numbers $1, \cos \left(r_1 \pi\right)$ and $\cos \left(r_2 \pi\right)$ are $\mathbb{Q}$-linearly independent.
\item $N\left(r_j\right) \geq 4$ for $j \in\{1,2\}$, and $\left(N\left(r_1\right), N\left(r_2\right)\right) \neq(5,5)$.
\end{enumerate}
\end{lem}

\begin{thm} \label{0331-2}
Let $X = X(\MB{Z}_{2l}, \{\pm a, l \} )$ be a connected $3$-regular graph,
where $a \in \{1, \dots, l-1\}$.
Then, perfect state transfer does not occur on $X$ between vertex type states.
\end{thm}

\begin{proof}
Let $x,y \in \MB{Z}_{2l}$ be distinct vertices in $X$.
We suppose that perfect state transfer occurs from $d^*e_x$ to $d^*e_y$ at time $\tau$, and want to derive a contradiction.
Let $P = P(X)$.
By Lemma~\ref{1206-3},
the eigenvalue $\lambda_1 = \frac{1}{3}(2\cos \frac{a}{l}\pi -1)$ of $P$ is in the eigenvalue support $\Theta_{P}(e_x)$.
Thus, Theorem~\ref{pst} implies that there exists an integer $s$ such that 
\begin{equation}\label{1208-2}
\frac{1}{3}(2\cos \frac{a}{l}\pi -1) = \cos \frac{s}{\tau}\pi.
\end{equation}
\TB{In particular, this means that the elements $1$, $\cos \frac{a}{l} \pi$ and $\cos \frac{a}{\tau}\pi$ are not $\mathbb{Q}$-linearly independent.}\MNOTE{Change A.12}

To apply Lemma~\ref{1208-1},
we want to check that neither $\frac{a}{l} + \frac{s}{\tau}$ nor $\frac{a}{l} - \frac{s}{\tau}$ is an integer.
We suppose that either $\frac{a}{l} + \frac{s}{\tau}$ or $\frac{a}{l} - \frac{s}{\tau}$ is an integer.
Then we have $\cos \frac{a}{l}\pi = \pm \cos \frac{s}{\tau}\pi$.
If $\cos \frac{a}{l}\pi = - \cos \frac{s}{\tau}\pi$,
then Equality~(\ref{1208-2}) if and only if $\cos \frac{s}{\tau}\pi = -\frac{1}{5}$.
This contradicts $\cos \MB{Q}\pi \cap \MB{Q} = \{\pm 1, \pm \frac{1}{2}, 0 \}$ which is asserted in Lemma~\ref{0115-2}.
Thus, we have $\cos \frac{a}{l}\pi \TR{=}\TB{\neq}\MNOTE[green]{Change~B.18} - \cos \frac{s}{\tau}\pi$.
Then Equality~(\ref{1208-2}) if and only if $\cos \frac{a}{l}\pi = -1$.
In particular, we have $a \in l\MB{Z}$,
but this is inconsistent with the way $a$ is taken.
We are therefore in a situation to apply Lemma~\ref{1208-1}, and hence either $N(\frac{a}{l}) < 4$, $N(\frac{s}{\tau}) < 4$, or $N(\frac{a}{l}) = N(\frac{s}{\tau}) = 5$ happens.

First, we want to reject $N(\frac{a}{l}) < 4$.
We assume $N(\frac{a}{l}) < 4$.
Since $a$ and $l$ are coprime,
we have $l = N(\frac{a}{l}) < 4$, i.e., $l \leq 3$.
If $l = 2$, then $X=X(\MB{Z}_{4}, \{\pm 1, 2\})$ is isomorphic to the complete graph $K_4$ with $4$ vertices.
However, $K_4$ does not admit perfect state transfer.
See Section~4.3 in~\cite{kubota2022perfect} for details.
Let $l = 3$.
If $a = 1$, then $X=X(\MB{Z}_{6}, \{\pm 1, 3\})$ is isomorphic to the complete bipartite graph $K_{3,3}$.
It has also been shown in Lemma~4.2 of \cite{kubota2022perfect} that this graph does not admit perfect state transfer.
If $a = 2$, then we have by~\eqref{1208-2} $\cos\frac{s}{\tau}\pi = \frac{-1+2\cos\frac{2}{3}\pi}{3} = -\frac{2}{3}$.
However, this contradicts Lemma~\ref{0115-2}

Next, we want to reject $N(\frac{s}{\tau}) < 4$.
We assume $N(\frac{s}{\tau}) < 4$.
Then we have $\cos \frac{s}{\tau} \pi \in \MB{Q}$ by Lemma~\ref{0115-2},
so from Equality~(\ref{1208-2}) we obtain $\cos \frac{a}{l} \pi \in \MB{Q}$.
Thus we have $N(\frac{a}{l}) < 4$ by Lemma~\ref{0115-2}.
The situation has been returned to the case $N(\frac{a}{l}) < 4$,
which has already been rejected.

Finally, 
we want to reject the case $N(\frac{a}{l}) = N(\frac{s}{\tau}) = 5$.
Then we have
\[ \cos\frac{a}{l}\pi, \cos\frac{s}{\tau}\pi \in \left\{ \frac{\pm 1 \pm \sqrt{5}}{4} \right\}. \]
However, it can be seen by inspecting the coefficient of $\sqrt{5}$ that none of the values satisfy Equality~(\ref{1208-2}).
\end{proof}

\section{Circulant graphs with valency $4$ and $a+b = l$}

Let $X=X(\MB{Z}_{2l}, \{\pm a, \pm b\})$ be a $4$-regular circulant graph.
Zhan \cite{zhan2019infinite} proved that $X$ admits perfect state transfer between vertex type states when $l$ is odd and $a+b=l$.
We completely characterize $X$ which admits perfect state transfer between vertex type states \TB{for every $l$}\MNOTE[green]{Change B.19}.
First, we consider the case of $l \not\equiv 0 \pmod 4$ and $a+b = l$.

\begin{lem} \label{1214-4}
Let $X=X(\MB{Z}_{2l}, \{\pm a, \pm b\})$ be a $4$-regular circulant graph.
Assume that $a + b = l$ with $l \not\equiv 0 \pmod 4$.
\begin{enumerate}
	\item If $l \equiv 2 \pmod 4$, then $X$ admits perfect state transfer from $d^*e_x$ to $d^*e_{x+l}$ at time $l$. \label{1214-4:2}
 	\item If $l$ is odd, then $X$ admits perfect state transfer from $d^*e_x$ to $d^*e_{x+l}$ at time $2l$.	\label{1214-4:1}
\end{enumerate}
Moreover, if $X$ is connected, then these times at which perfect state transfer occurs are minimum.
\end{lem}
\begin{proof}
First, we prove the case~\ref{1214-4:2}.
Let $l \equiv 2 \pmod 4$.
By Lemma~\ref{1212-1}~\ref{1214-4:1} and Theorem~\ref{1124-1},
we may assume that $x=0$ without loss of generality.
Let $P = P(X)$.
%We therefore prove the case~\ref{1214-4:2}.
%By Lemma~\ref{1212-1}~\ref{1214-4:1} and Theorem~\ref{1124-1},
%we may assume that $x=0$ without loss of generality.
%Let $P = P(X)$.
%To prove~\ref{1214-4:2},
We will check the three conditions \ref{pst:C:a}--\ref{pst:C:c} of~\ref{pst:C} in Theorem~\ref{pst}.
Before that, we calculate eigenvalues and some projections.
The eigenvalues of $P$ are
\begin{equation} \label{1213-1}
\lambda_j = \begin{cases}
0 \quad &\text{if $j$ is odd,} \\
\cos\frac{aj}{l}\pi \quad &\text{if $j$ is even}
\end{cases}
\end{equation}
for $j \in \TR{[2l-1] }\TB{[2l]}$\MNOTE{Change A.13}, because
\begin{align*}
\lambda_j &= \frac{\zeta_{2l}^{aj} + \zeta_{2l}^{-aj} + \zeta_{2l}^{bj} + \zeta_{2l}^{-bj}}{4} \tag{by Lemma~\ref{lem:ortho}} \\
&= \frac{2\cos \frac{aj}{l}\pi + 2\cos \frac{bj}{l}\pi}{4} \\
&= \frac{2\cos \frac{aj}{l}\pi + 2\cos \frac{(l-a)j}{l}\pi}{4} \tag{by $a+b=l$} \\
&= \frac{1+(-1)^j}{2} \cos \frac{aj}{l}\pi.
\end{align*}
Next we verify
\begin{align}	\label{1213-1:1}
	E_\lambda e_0	= \begin{cases}	
		-E_\lambda e_l & \text{ if } \lambda = 0 \\
		E_\lambda e_l & \text{ otherwise }.
	\end{cases}
\end{align}
Recall the vectors $u_j$ defined by~(\ref{1212-2}).
We have
\begin{equation} \label{uj0}
\Span{u_j, e_0} = \zeta_{2l}^0 = 1
\end{equation}
and
\begin{equation} \label{ujl}
\Span{u_j, e_l} = \zeta_{2l}^{lj} = \zeta_{2}^j = (-1)^j.
\end{equation}
Let 
$\MC{I}_{\lambda} = \{ j \in [2l] \mid \lambda_j = \lambda \}$ for a real number $\lambda$.
%\MNOTE[green]{Change~B.20}
%\TB{
%\begin{align}
%    \MC{I}_{\lambda} = \{ j \in [2l] \mid \lambda_j = \lambda \} \text{ for a real number } \lambda.
%\end{align}
%}
We check
\begin{equation} \label{1213-2}
	\MC{I}_{0} = \{1,3, \dots, 2l-1\}.
\end{equation}
By~(\ref{1213-1}), we immediately have $\MC{I}_0 \supset \{1,3, \dots, 2l-1\}$.
We suppose that an even integer $j$ satisfies $\lambda_j = 0$.
Then by~(\ref{1213-1}) we have $0 = \lambda_j = \cos\frac{aj \pi}{l}$, so there exists an odd integer $s$ such that $\frac{aj \pi}{l} = \frac{s}{2}\pi$.
Thus, we obtain
$ j = \frac{ls}{2a} = \frac{\frac{l}{2}s}{a}.$
Since $l \equiv 2 \pmod 4$,
both $\frac{l}{2}$ and $s$ are odd,
which contradicts that $j$ is even.
Therefore, we have $\MC{I}_{0} = \{1,3, \dots, 2l-1\}$.
Let $\lambda \in \spec(P)$, and calculate $E_\lambda e_0$.
If $\lambda = 0$, then
\begin{align*}
\TR{n}\TB{2l}\MNOTE{Change A.14} \cdot E_0 e_0 &= \sum_{j \in \MC{I}_0}u_j u_j^* e_0 \tag{by (\ref{1225-1})} \\
&= \sum_{j \in \MC{I}_0}u_j \cdot 1 \tag{by (\ref{uj0})} \\
&= -\sum_{j \in \MC{I}_0}u_j \cdot (-1) \\
&= -\sum_{j \in \MC{I}_0}u_j u_j^* e_l \tag{by (\ref{ujl}) and (\ref{1213-2})} \\
&= -\TR{n}\TB{2l}\MNOTE{Change A.14} \cdot E_0 e_l. \tag{by (\ref{1225-1})}
\end{align*}
On the other hand, if $\lambda \neq 0$,
then from Equality~(\ref{1213-2}) it is guaranteed that the indices in the set $[2l] \setminus \MC{I}_0$ are even.
Thus we have
\begin{align*}
\TR{n}\TB{2l}\MNOTE{Change A.14} \cdot E_{\lambda} e_0 &= \sum_{j \in \MC{I}_{\lambda}}u_j u_j^* e_0 \tag{by (\ref{1225-1})} \\
&= \sum_{j \in \MC{I}_{\lambda}} u_j \cdot 1 \tag{by (\ref{uj0})} \\
&= \sum_{j \in \MC{I}_{\lambda}} u_j u_j^* e_l \tag{by (\ref{ujl})} \\
&= \TR{n}\TB{2l}\MNOTE{Change A.14} \cdot E_{\lambda} e_l.
\end{align*}

Next, we check the three conditions \ref{pst:C:a}--\ref{pst:C:c} of~\ref{pst:C} in Theorem~\ref{pst} with $\gamma=1$ and $\tau=l$.
First, \ref{pst:C:a} follows from~\eqref{1213-1:1}.
Next we check~\ref{pst:C:b}.
Note that $\Theta_P(e_0) = \spec(P)$ holds by Lemma~\ref{1206-3}.
Let $\lambda \in \spec(P)$, and assume $E_\lambda e_0 =  E_\lambda e_l$.
Then, $\lambda \neq 0$ follows from~\eqref{1213-1:1}, and this together with~\eqref{1213-2} implies that $\lambda = \lambda_j$ holds for some even $j$.
From~\eqref{1213-1}, we see that \ref{pst:C:b} is satisfied.
We check~\ref{pst:C:c}. 
Assume $E_\lambda e_0 = - E_\lambda e_l$.
Then, by~\eqref{1213-1:1}, we have
$
	\lambda = 0 = \cos \frac{\pi}{2} = \cos \frac{\frac{l}{2}}{l}\pi.
$
Since $l/2$ is odd, we see \ref{pst:C:c} is satisfied. 
Therefore, by Theorem~\ref{pst}, we see that $X$ admits perfect state transfer from $d^*e_0$ to $d^*e_{l}$ at time $l$.

Finally, we show minimality of the time at which perfect state transfer occurs.
By Theorem~\ref{pst}~\ref{pst:C}, if perfect state transfer occurs at time $\tau$, then $T_{\tau}(\lambda) = \pm 1$ for any $\lambda \in \Theta_{P}(e_0)$.
%This is claimed in Theorem~3.3 of \cite{kubota2022perfect}.
By~(\ref{1213-1}),
we have $0 \in \spec(P) = \Theta_{P}(e_0)$,
so $\pm 1 = T_{\tau}(0) = T_{\tau}(\cos\frac{\pi}{2}) = \cos\frac{\tau}{2}\pi$,
and hence $\frac{\tau}{2} \in \MB{Z}$,
i.e.,
\begin{equation} \label{1214-1}
\tau \in 2\MB{Z}.  
\end{equation}
In addition,
since $l \equiv 2 \pmod 4$,
$\frac{l}{2} -1$ is even integer.
Thus, (\ref{1213-1}) implies that
$ \cos \frac{a(\frac{l}{2} - 1)\pi }{l} \in \spec(P),$
and hence
\[ \pm 1 = T_{\tau}\left( \cos \frac{a(\frac{l}{2} - 1)}{l}\pi \right) = \cos \frac{a \tau (\frac{l}{2} - 1)}{l}\pi. \]
We have
$\frac{a \tau (\frac{l}{2} - 1)}{l} \in \MB{Z},$
so
$ a\tau\left( \frac{l}{2} - 1 \right) \in l\MB{Z} \subset \frac{l}{2}\MB{Z}.$
Since $a$ and $\frac{l}{2}$ are coprime by the connectivity of $X$, and $\frac{l}{2} - 1$ and $\frac{l}{2}$ are coprime,
we have
$
\tau \in \frac{l}{2}\MB{Z}.    
$
Hence, by (\ref{1214-1}),
we obtain $\tau \in l\MB{Z}$, i.e., $\tau \geq l$.
The minimum time at which perfect state transfer occurs is $l$.

The case~\ref{1214-4:1} has already been shown in Theorem~6.1 of \cite{zhan2019infinite}.
We can show that the time $l$ is minimum in the same manner in~\ref{1214-4:2}.
\end{proof}

In contrast to the above lemma,
perfect state transfer does not occur if $l \equiv 0 \pmod 4$.

\begin{lem}
Let $X=X(\MB{Z}_{2l}, \{\pm a, \pm b\})$ be a connected $4$-regular graph.
Assume that $a + b = l$.
If $l \equiv 0 \pmod 4$,
then $X$ does not admit perfect state transfer between vertex type states.
\end{lem}

\begin{proof}
As confirmed in Lemma~\ref{1214-3},
if perfect state transfer occurs from $d^*e_x$ to $d^*e_y$, then $y = x+l$.
By Lemma~\ref{1212-1}~(i) and Theorem~\ref{1124-1},
we may assume that $x=0$ without loss of generality.
We check that the first condition of Theorem~\ref{pst}~(C) does not hold.
We use the same notation as in the proof of Lemma~\ref{1214-4}.
\TR{Since $X$ is connected, $a$ and $b$ are odd.
Hence, $\frac{l}{2} \in \MC{I}_0$.}\MNOTE{Change A.15}
\TB{
Since $X$ is connected, either $a$ or $b$ is odd.
Furthermore, by the assumptions $a+b = l$ and $l \equiv 0 \pmod 4$, both $a$ and $b$ are odd.
Noting that $\frac{l}{2}$ is even, we have by~\eqref{1213-1},
\begin{align*}  
    \lambda_1 = 0 
    \quad \text{ and } \quad
    \lambda_{\frac{l}{2}} 
    = \cos \left( \frac{a \cdot \frac{l}{2}}{l} \pi\right)
    = \cos \frac{a}{2}\pi = 0,
\end{align*}
and obtain \MNOTE[green]{Change B.20}
\begin{align}  \label{1022}
    1, \frac{l}{2} \in \MC{I}_0 = \{ j \in [2l] \mid \lambda_j = 0 \}.
\end{align}
}
By Equalities~(\ref{uj0}) and~(\ref{ujl}),
we have $\TB{2l \cdot}\MNOTE[black]{Change C.4}E_0e_0 = \sum_{j \in \MC{I}_0} u_j$ and $\TB{2l \cdot} E_0e_l = \sum_{j \in \MC{I}_0} (-1)^j u_j$.
Supposing that the first condition of Theorem~\ref{pst}~(C) holds,
from linearly independence of eigenvectors we obtain either $1 = (-1)^j$ for any $j \in \MC{I}_0$,
or $1 = -(-1)^j$ for any $j \in \MC{I}_0$.
However, by~\eqref{1022} and $\frac{l}{2} \in 2\MB{Z}$,
none of the above is valid.
We see that perfect state transfer does not occur.
\end{proof}

\section{Circulant graphs with valency $4$ and $a+b \neq l$}
%do not admit perfect state transfer
%The case $a+b \neq l$
Following the previous section,
we study $4$-regular circulant graphs $X=X(\MB{Z}_{2l}, \{\pm a, \pm b\})$ with $a+b \neq l$.
To conclude first, such circulant graphs do not admit perfect state transfer.
However, its proof requires a lot of effort.
To begin, we check obvious conditions required by the circulant graph $X$ with valency $4$.
For the graph $X$ to have valency $4$,
the following must be satisfied:
\begin{align}
a,b &\not\equiv 0 \pmod{l}, \label{0128-2} \\
a \pm b &\not\equiv 0 \pmod{2l}. \label{0128-3}
\end{align}
Furthermore, we may assume that $a,b \equiv 1,2,\dots, l-1 \pmod{2l}$ without loss of generality, and since $a-b \equiv \pm (l-2), \dots, \pm 1 \pmod{2l}$, we must have
\begin{equation}
a - b \not\equiv l \pmod{2l}. \label{0128-4}
\end{equation}

\subsection{Necessary conditions for a graph to admit perfect state transfer}

A complex number $\alpha$ is said to be an \emph{algebraic integer}
if there exists a monic polynomial $p(x) \in \MB{Z}[x]$ such that $p(\alpha) = 0$.
Let $\Omega$ denote the set of algebraic integers.
The set of algebraic integers forms a ring,
in particular, $\Omega$ is closed under addition and multiplication.
A specific example is that $\zeta_m^k$ and $\zeta_m^{-k}$ are algebraic integers because $(\zeta_m^k)^m = (\zeta_m^{-k})^m = 1$.
Hence, for a rational number $2k/m$, $2\cos \frac{2k}{m}\pi = \zeta_m^k + \zeta_m^{-k}$ is also an algebraic integer. 

Let $K$ be a finite extension field of the field $\MB{Q}$ of rational numbers.
The set $\{ \omega_1, \dots, \omega_t \}$ is an \emph{integral basis} for $K$ when every element of $K \cap \Omega$ is uniquely expressible as a $\MB{Z}$-linear combination of elements of the set.
For example, an integral basis for $\MB{Q}$ is $\{1\}$,
i.e., $\MB{Q} \cap \Omega = \MB{Z}$,
and an integral basis for $\MB{Q}(\sqrt{2})$ is $\{1, \sqrt{2}\}$, i.e., $\MB{Q}(\sqrt{2}) \cap \Omega = \{ p+q\sqrt{2} \mid p,q \in \MB{Z} \}$.
Readers who want to know more details on integral bases should refer to \cite{jarvis2014algebraic}.

\begin{lem} \label{0128-5}
Let $\G$ be a graph with the discriminant $P$,
and let $x,y$ be vertices of $\G$.
If perfect state transfer occurs from $d^*e_x$ to $d^*e_y$ at time $\tau$,
then $2\lambda$ is an algebraic integer for any $\lambda \in \Theta_{P}(e_x)$.
\end{lem}

\begin{proof}
Assume perfect state transfer occurs from $d^*e_x$ to $d^*e_y$ at time $\tau$. 
Let $\lambda \in \Theta_P(e_x)$.
By Theorem~\ref{pst},
we have $\lambda = \cos \frac{j}{\tau}\pi$ for some integer $j$.
Hence $2\lambda$ is an algebraic integer.
\end{proof}

\begin{lem}	\label{nonalg}
	Let $l \in \MB{Z}_{\geq 1}$.
	If a $4$-regular circulant graph $X = X(\MB{Z}_{2l}, \{ \pm a, \pm b\})$ admits perfect state transfer, then
	\begin{align}	\label{nonalg:1}
		\frac{\zeta_{2l}^a + \zeta_{2l}^{-a} + \zeta_{2l}^b + \zeta_{2l}^{-b} }{2}
	\end{align}
	is an algebraic integer.
\end{lem}
\begin{proof}
From Lemmas~\ref{lem:ortho} and~\ref{1206-3},
we have $(\zeta_{2l}^a + \zeta_{2l}^{-a} + \zeta_{2l}^b + \zeta_{2l}^{-b} )/4 \in \spec(P(X)) = \Theta_{P(X)}(e_0)$.
%From Lemma~\ref{1206-3}, we have $\Theta_{P(X)}(e_0) = \spec(P(X))$.
%From Lemma~\ref{lem:ortho}, $(\zeta_{2l}^a + \zeta_{2l}^{-a} + \zeta_{2l}^b + \zeta_{2l}^{-b} )/4 \in \spec(P(X)) = \Theta_{P(X)}(e_0)$.
Therefore, Lemma~\ref{0128-5} implies that \TB{if}\MNOTE{Change A.16} $X$ admits perfect state transfer, then~\eqref{nonalg:1} is an algebraic integer.
\end{proof}

\begin{lem} \label{0114-1}
Let $l$ be an odd integer,
and let $c$ be an integer.
Then we have $\zeta_{2l}^c = (-1)^c \zeta_l^{\varepsilon(c)}$,
where 
\[ \varepsilon(c) = \begin{cases}
	\frac{c}{2}  &\text{if $c$ is even,} \\
	\frac{c+l}{2}  &\text{if $c$ is odd.}
\end{cases} \]
In particular,
if a $4$-regular circulant graph $X = X(\MB{Z}_{2l}, \{ \pm a, \pm b\})$ admits perfect state transfer, then
\begin{align}
\frac{\zeta_{2l}^a + \zeta_{2l}^{-a} + \zeta_{2l}^b + \zeta_{2l}^{-b} }{2} \notag = 
\frac{(-1)^a \zeta_{l}^{\varepsilon(a)} + (-1)^a \zeta_{l}^{-\varepsilon(a)} + (-1)^b \zeta_{l}^{\varepsilon(b)} + (-1)^b \zeta_{l}^{-\varepsilon(b)} }{2} \label{0128-1}
\end{align}
is an algebraic integer.
\end{lem}

\begin{proof}
If $c$ is even,
then we have $\zeta_{2l}^c = \zeta_{l}^{\frac{c}{2}}$.
If $c$ is odd,
then we have $\zeta_{2l}^c = - \zeta_{2l}^l \zeta_{2l}^c = -\zeta_{2l}^{c+l} = -\zeta_{l}^{\frac{c+l}{2}}$.
Hence these together with Lemma~\ref{0128-5} implies the desired result.
\end{proof}

By this lemma, what we need to discuss is whether
\begin{equation} \label{0120-2}
\Delta = \frac{\pm \zeta_{m}^{p} \pm \zeta_{m}^{-p} \pm \zeta_{m}^{q} \pm \zeta_{m}^{-q}}{2}
\end{equation}
is an algebraic integer for some positive integers $m$, $p$ and $q$.
%such that $m \not\equiv 2 \pmod 4$.
To do so, we once forget the condition $a+b \neq l$ for the exponents with respect to $\zeta_{2l}$
and study a sufficient condition for $\Delta$ not to be an algebraic integer.
Readers who want to first know its sufficient condition should refer to Lemma~\ref{0120-1}.
Due to circumstances of this sufficient condition,
we first need to study $\MB{Q}$-linearly independence of $\zeta_m^{p}, \zeta_m^{-p}, \zeta_m^{q}, \zeta_m^{-q}$.

\subsection{$\MB{Q}$-linearly independence}
The following is easily verified.

\begin{lem} \label{0115-3}
Let $s,t$ be integers such that $s \pm t \not\equiv 0 \pmod 5$.
Then $\cos \frac{s}{5}\pi$ and $\cos \frac{t}{5}\pi$ are $\MB{Q}$-linearly independent.
\end{lem}

\begin{lem} \label{0115-1}
Let $m$ be a positive integer,
and let $p$ be an integer.
Then $4p \equiv 0 \pmod m$ if and only if $N(\frac{2p}{m}) \in \{1,2\}$.
\end{lem}

\begin{proof}
First, we suppose that $4p \equiv 0 \pmod m$.
There exists an integer $k$ such that $4p = km$,
i.e., $\frac{2p}{m} = \frac{k}{2}$.
Thus, we have $N(\frac{2p}{m}) = N(\frac{k}{2}) \leq 2$.
Conversely, we suppose that $N(\frac{2p}{m}) \in \{1,2\}$.
If $N(\frac{2p}{m}) \in \{1, 2\}$,
then there exists an integer $s$ such that $\frac{2p}{m} = \frac{s}{2}$,
and hence $4p = ms \equiv 0 \pmod m$.
\end{proof}

\begin{lem} \label{0115-4}
Let $m$ be a positive integer,
and let $p, q$ be distinct integers.
If $4p, 4q, 2(p \pm q) \not\equiv 0 \pmod m$,
then $\cos \frac{2p}{m}\pi$ and $\cos \frac{2q}{m}\pi$ are $\MB{Q}$-linearly independent.
\end{lem}

\begin{proof}
By Lemma~\ref{0115-1},
we have $N(\frac{2p}{m}), N(\frac{2q}{m}) \geq 3$.
First, we suppose that $N(\frac{2p}{m}) = N(\frac{2q}{m}) = 3$.
Then, there exist $s,t \in \MB{Z} \setminus 3\MB{Z}$ such that $\frac{2p}{m} = \frac{s}{3}$ and $\frac{2q}{m} = \frac{t}{3}$.
Since $\frac{m}{3}(s+t) = 2(p+q) \not\equiv 0 \pmod m$,
either $s \equiv t \equiv 1 \pmod 3$ or $s \equiv t \equiv 2 \pmod 3$ has to be held.
However, in both cases we have $2(p-q) = \frac{m}{3}(s-t) \equiv 0 \pmod m$, which contradicts our assumption.
Therefore, either $N(\frac{2p}{m})$ or $N(\frac{2q}{m})$ is greater than $3$.
Without loss of generality, we may assume that %$N(\frac{2p}{m}) \geq 3$ and
$N(\frac{2q}{m}) \geq 4$.

First, we consider the case $N(\frac{2p}{m}) = 3$.
Lemma~\ref{0115-2} implies that $\cos \frac{2p}{m}\pi$ is a rational number while $\cos \frac{2q}{m}\pi$ is an irrational number.
Moreover, since $N(\frac{2p}{m}) = 3$, we have $\cos \frac{2p}{m}\pi \neq 0$.
Thus, we see that $\cos \frac{2p}{m}\pi$ and $\cos \frac{2q}{m}\pi$ are $\MB{Q}$-linearly independent.

Next, we consider the case $N(\frac{2p}{m}) \geq 4$.
If $(N(\frac{2p}{m}), N(\frac{2q}{m})) \neq (5,5)$,
then Lemma~\ref{1208-1} implies the conclusion.
If $(N(\frac{2p}{m}), N(\frac{2q}{m})) = (5,5)$,
then there exist $s,t \in \MB{Z} \setminus 5\MB{Z}$ such that $\frac{2p}{m} = \frac{s}{5}$ and $\frac{2q}{m} = \frac{t}{5}$.
By our assumption, we have $s \pm t \not\equiv 0 \pmod 5$
since
\begin{align*}
2(p \pm q) \not\equiv 0 \pmod m
&\iff 10(p \pm q) \not\equiv 0 \pmod{5m} \\
&\iff m(s \pm t) \not\equiv 0 \pmod{5m} \\
&\iff s \pm t \not\equiv 0 \pmod 5.
\end{align*}
Therefore, Lemma~\ref{0115-3} leads to the conclusion.
\end{proof}

\begin{lem} \label{0115-6}
Let $m$ be a positive integer,
and let $p, q$ be distinct integers.
If $4p, 4q, 2(p \pm q) \not\equiv 0 \pmod m$,
then $\sin \frac{2p}{m}\pi$ and $\sin \frac{2q}{m}\pi$ are $\MB{Q}$-linearly independent.
\end{lem}

\begin{proof}
Since $\sin \frac{2p}{m}\pi = \cos \frac{2(m-4p)}{4m} \pi$ and $\sin \frac{2q}{m}\pi = \cos \frac{2(m-4q)}{4m} \pi$,
it is sufficient to confirm that Lemma~\ref{0115-4} is applicable, i.e.,
$4(m-4p), 4(m-4q), 4(m-2p-2q), 8(p-q) \not\equiv 0 \pmod{4m}$.
This can be verified straightforwardly.
\end{proof}

\begin{lem} \label{0127-1}
Let $m$ be a positive integer,
and let $p, q$ be distinct integers.
If $4p, 4q, 2(p \pm q) \not\equiv 0 \pmod m$,
then the numbers $\zeta_m^{p}, \zeta_m^{-p}, \zeta_m^{q}, \zeta_m^{-q}$ are $\MB{Q}$-linearly independent.
%\textcolor{red}{
%Let $m$ be a positive integer,
%and let $p, q$ be distinct integers.
%The numbers $\zeta_m^{p}, \zeta_m^{-p}, \zeta_m^{q}, %\zeta_m^{-q}$ are $\MB{Q}$-linearly independent if and only if $4p, 4q, 2(p \pm q) \not\equiv 0 \pmod m$.}
\end{lem}

\begin{proof}
We assume that a $\MB{Q}$-linear combination
\begin{equation} \label{0115-5}
c_1\zeta_m^{p} + c_2\zeta_m^{-p} +c_3 \zeta_m^{q} +c_4 \zeta_m^{-q} = 0.
\end{equation}
Take the conjugate of both sides and add it to Equality~(\ref{0115-5}).
We have
\begin{align*}
0 &= (c_1 + c_2)(\zeta_m^{p} + \zeta_m^{-p}) + (c_3 + c_4)(\zeta_m^{q} + \zeta_m^{-q}) \\
&= 2(c_1 + c_2) \cos \frac{2p}{m}\pi + 2(c_3 + c_4) \cos \frac{2q}{m}\pi.
\end{align*}
By Lemma~\ref{0115-4},
we have $c_1 + c_2 = 0$ and $c_3 + c_4 = 0$.
Substituting them into Equality~(\ref{0115-5}),
we have
\begin{align*}
0 &= c_1\zeta_m^{p} - c_1\zeta_m^{-p} +c_3 \zeta_m^{q} -c_3 \zeta_m^{-q} \\
&= c_1(\zeta_m^{p} - \zeta_m^{-p}) + c_3(\zeta_m^{q} - \zeta_m^{-q}) \\
&= 2i c_1 \sin \frac{2p}{m}\pi + 2i c_3 \sin \frac{2q}{m}\pi.
\end{align*}
Lemma~\ref{0115-6} implies that $c_1 = c_3 = 0$.
Therefore, we have $c_1 = c_2 = c_3 = c_4 = 0$,
and hence we see that $\zeta_m^{p}, \zeta_m^{-p}, \zeta_m^{q}, \zeta_m^{-q}$ are $\MB{Q}$-linearly independent.
\end{proof}

\subsection{Canonical bases of cyclotomic fields}
For a set $A$ and a positive integer $k$,
we denote by $\C{A}{k}$ the set of all $k$-subsets of $A$, i.e.,
\[ \C{A}{k} = \{ S \mid S \subset A, |S| = k \}. \]
For a prime number $p$,
we define
\[
p^* = \begin{cases}
4 \, &\text{if $p=2$,} \\
p \, &\text{otherwise.}
\end{cases}
\]
\TB{Let $\varphi$ denote Euler's totient function, i.e., $\varphi(n) = |\mathbb{Z}_n^\times|$ for a positive integer $n$.}\MNOTE[green]{Change B.21}
The following is on integral bases of the cyclotomic fields,
which plays an important role in this section.
We cite a result by Bosma, but slightly change notation.

\begin{thm}[Theorem~5.1 in \cite{bosma1990canonical}] \label{0117-1}
Let $n$ be a positive positive integer with $n \not\equiv 2 \pmod 4$,
and let $n = p_1^{f_1} p_2^{f_2} \cdots p_t^{f_t}$ be the prime factorization of $n$.
For each prime factor $p_j$,
let $A_{p_j} \in \C{[p_j]}{\varphi(p_j^*)}$,
and $B_{p_j} = \left[ \frac{p_j^{f_j}}{p_j^*} \right]$.
Then the set
\begin{align}	\label{0117-1:1}
I  = \left\{ \prod_{j=1}^t \left( \zeta_{p_j^*}^{a_j} \zeta_{p_j^{f_j}}^{b_j} \right)
\, \middle| \,
a_1 \in A_{p_1}, \dots, a_t \in A_{p_t},
b_1 \in B_{p_1}, \dots, b_t \in B_{p_t}
\right\}
\end{align}
forms a basis for $\MB{Q}(\zeta_n)$ over $\MB{Q}$.
Moreover, for every divisor $d$ of $n$,
the set $I_d = I \cap \MB{Q}(\zeta_d)$
forms an integral basis for $\MB{Q}(\zeta_d)$.
\end{thm}

\begin{ex}
We take $n = 36 = 2^2 \cdot 3^2$ as an example to help understand Theorem~\ref{0117-1}.
Since the $\MB{Q}$-minimal polynomial of $\zeta_{36}$ is $x^{12}-x^6+1$,
we can take $1, \zeta_{36}, \zeta_{36}^2, \dots, \zeta_{36}^{11}$ as a $\MB{Q}$-basis of $\MB{Q}(\zeta_{36})$.
On the other hand, 
we can take another basis of $\MB{Q}(\zeta_{36})$ using Theorem~\ref{0117-1}.
The sets $A_2 = \C{[2]}{\varphi(2^*)} = \{0,1\}$,
$B_2 = \left[ \frac{2^{2}}{2^*} \right] = \{0\}$,
and $B_3 = \left[ \frac{3^{2}}{3^*} \right] = \{0,1,2\}$
are automatically determined, while the set $A_3 \in \C{[3]}{\varphi(3^*)}$ can be chosen arbitrarily, for example $A_3  = \{0,2\}$.
Then,
\begin{align*}
I &= \left\{ \zeta_{2^*}^{a_2} \zeta_{2^{2}}^{b_2} \zeta_{3^*}^{a_3} \zeta_{3^{2}}^{b_3} \, \middle| \, a_2 \in A_2, a_3 \in A_3,
b_2 \in B_2, b_3 \in B_3\right\} \\
&= \left\{ \zeta_{4}^{a_2} \zeta_{4}^{b_2} \zeta_{3}^{a_3} \zeta_{9}^{b_3} \, \middle| \, a_2 \in \{0,1\}, a_3 \in \{0,2\},
b_2 \in \{0\}, b_3 \in \{0,1,2\} \right\} \\
&= \{ \zeta_{36}^0, \TB{\zeta_{36}^{1},} \zeta_{36}^{4}, \TB{\zeta_{36}^{5},} \zeta_{36}^{8}, \zeta_{36}^{9},
\zeta_{36}^{13}, \zeta_{36}^{17}, \TR{\zeta_{36}^{21},} \zeta_{36}^{24}, \TR{\zeta_{36}^{25},} \zeta_{36}^{28} \TR{, \zeta_{36}^{29}}, \zeta_{36}^{32} \TB{, \zeta_{36}^{23}}\MNOTE[green]{Change~B.22}
\}
\end{align*}
forms a $\MB{Q}$-basis for $\MB{Q}(\zeta_{36})$ by Theorem~\ref{0117-1}.
Furthermore, $I_n = I \cap \MB{Q}(\zeta_{36}) = I$ forms an integral basis for $\MB{Q}(\zeta_{36})$. 
\end{ex}

\subsection{Decompositions of $\zeta_n^x$}
Display the prime factorization of an integer $n \geq 2$ as $n = p_1^{f_1} p_2^{f_2} \cdots p_t^{f_t}$,
where $p_1, \ldots, p_t$ are pairwise distinct prime numbers, and $f_1, \ldots, f_t$ are positive integers.
Then, we define the mapping $\Phi_n: \prod_{j=1}^t [p_j^{f_j}] \to \{ \zeta_n^{j} \mid j \in [n] \}$ by
\[ \Phi_n(x_1, x_2, \dots, x_t)= \prod_{j=1}^t \zeta_{p_j^{f_j}}^{x_j}. \]

\begin{lem} \label{0117-2}
The mapping $\Phi_n$ is bijective for any integer $n\geq 2$.
\end{lem}

\begin{proof}
Since the number of elements in the domain of $\Phi_n$ and that in the codomain are equal,
it is sufficient to show that $\Phi_n$ is surjective.
Take $\zeta_n^m \in \{ \zeta_n^{j} \mid j \in [n] \}$.
Let $n = p_1^{f_1} p_2^{f_2} \cdots p_t^{f_t}$ be the prime factorization of $n$.
The greatest common divisor of \TR{$np_1^{-f_1}$, $np_1^{-f_2}$, \dots, $np_1^{-f_t}$}\TB{$np_1^{-f_1}$, $np_2^{-f_2}$, \dots, $np_t^{-f_t}$}\MNOTE[green]{Change B.23} is $1$,
so there exist integers $x_1, x_2, \dots, x_t$ such that
%\[ \sum_{j=1}^t x_j \frac{n}{p_j^{f_j}} = m. \]
\[ m = \sum_{j=1}^t x_j (np_j^{-f_j})
= \sum_{j=1}^t n x_j p_j^{-f_j}. \]
Let $x_j'$ be the remainder of $x_j$ divided by $p_j^{f_j}$ for each $j$.
Then, we have $x_j' \in [p_j^{f_j}]$, and
\[
\Phi(x_1', x_2', \dots, x_t')
= \prod_{j=1}^t \zeta_{p_j^{f_j}}^{x_j'} \\
= \prod_{j=1}^t \zeta_{p_j^{f_j}}^{x_j} \\
= \prod_{j=1}^t \zeta_{n}^{n x_j p_j^{-f_j}} \\ %%
%= \prod_{j=1}^t \zeta_{n}^{x_j \cdot \frac{n}{p_j^{f_j}}} \\ %%
= \zeta_n^{\sum_{j=1}^t n x_j p_j^{-f_j}} \\
= \zeta_n^{m}.
\]
Therefore, we see that $\Phi_n$ is surjective.
\end{proof}

\begin{defi}	\label{phi}
Let $n \geq 2$ be an integer,
and let $n = p_1^{f_1} p_2^{f_2} \cdots p_t^{f_t}$ be the prime factorization of $n$,
where $f_1, \ldots, f_t$ are positive integers.
Since $\Phi_n$ is bijective by Lemma~\ref{0117-2},
let $\Psi_n$ be the inverse mapping of $\Phi_n$.
For $\zeta_n^x \in \{ \zeta_n^j \mid j \in [n] \}$,
we write
\begin{align*}
	\Psi_n(\zeta_n^x) = (x_1, x_2, \dots, x_t) \in [p_1^{f_1}] \times [p_2^{f_2}] \times \cdots \times [p_t^{f_t}].
\end{align*}
For each prime factor $p_j$, we define $\Psi_n^{(p_j)}(\zeta_n^x) := x_j$.
In the case of $p _j \neq 2$, we define
\begin{align*}
	\pi_n^{(p_j)}(\zeta_n^x) :=   x_j' 
	\quad \text{ and } \quad
	\theta_n^{(p_j)}(\zeta_n^x) :=   x_j'' 
\end{align*}
where $x_j' \in [p_j]$ and $x_j'' \in [p_j^{f_j-1}]$ such that $x_j = x_j' p_j^{f_j-1} + x_j''$.
In the case of $p_j = 2$ and $f_j \geq 2$,
define 
\begin{align*}
	\pi_n^{(2)}(\zeta_n^x) :=   x_j' 
	\quad \text{ and } \quad
	\theta_n^{(2)}(\zeta_n^x) :=   x_j'' 
\end{align*}
where
%$x_1' \in [2]$ and $x_1'' \in [2^{f_1-1}]$ satisfy $x_1' = x_1 \in [2]$ and $x_1''=0$ if $f_1 = 1$, and 
$x_j' \in [4]$ and $x_j'' \in [2^{f_j-2}]$ satisfy $x_j = x_j' 2^{f_j-2} + x_j''$. % if $f_1 \geq 2$.
\end{defi}

Note that we do not define $\Phi_n$ in the case of $n \equiv 2 \pmod 4$, i.e., $p_j=2$ and $f_j=1$
since this case is reduced to the case of $n \equiv 0  \pmod 4$ in proofs of our results.

\begin{ex}
We take $n := 36 = 2^2 \cdot 3^2$ and $x := 5$ as an example to help understand the mappings in Definition~\ref{phi}.
Then, since $\zeta_{36}^5 = \zeta_4^1 \zeta_9^8$,
we have $\Psi_{n}(\zeta_n^{x}) = (1, 8)$,
$\Psi_{n}^{(2)}(\zeta_n^{x}) = 1$ and
$\Psi_{n}^{(3)}(\zeta_n^{x}) = 8$.
Also, since $1 = 1 \cdot 2^{2-2} + 0$ and $8 = 2 \cdot 3^{2-1} + 2$,
we have $\pi_n^{(2)}(\zeta_n^x) = 1$, 
$\theta_n^{(2)}= 0$ 
$\pi_n^{(3)}(\zeta_n^x) = 2$, 
and $\theta_n^{(3)}(\zeta_n^x) = 2$.
\end{ex}

\TB{
The reason for defining the functions $\pi_m^{(p)}$ and $\theta_m^{(p)}$ is that they are useful tools for checking whether the element 
$$
    \Delta = \frac{\pm \zeta_{m}^{p} \pm \zeta_{m}^{-p} \pm \zeta_{m}^{q} \pm \zeta_{m}^{-q}}{2}
$$
in~\eqref{0120-2} of the cyclotomic field $\mathbb{Q}(\zeta_m)$ is an algebraic integer.
Indeed, the element $\Delta$ is an algebraic integer if and only if $\Delta$ can be expressed as a $\mathbb{Z}$-linear combination of an integral basis of the cyclotomic field $\mathbb{Q}(\zeta_m)$.
Therefore, we will construct an integral basis $I$ containing as many of $\zeta_m^{p}$, $\zeta_m^{-p}$, $\zeta_m^{q}$, $\zeta_m^{-q}$ as possible,
and then check whether $\Delta$ is a $\mathbb{Z}$-linear combination of $I$.
In fact, we can construct such an integral basis by Theorem~\ref{0117-1}, which asserts that $\zeta_m^x$ is contained in $I$ if $A_p$ and $B_p$ are chosen such that $\pi_m^{(p)}(\zeta_m^x) \in A_p$ and $\theta_m^{(p)}(\zeta_m^x) \in B_p$ for every prime factor $p$ of $m$.
This is the reason why we introduced $\pi_m^{(p)}$ and $\theta_m^{(p)}$.
}

\TB{
In the following subsection, we will determine whether $\Delta$ is an algebraic integer for some specific values of $m,p,q$.
Before proceeding, we prepare a lemma concerning $\pi_m^{(p)}$ and $\theta_m^{(p)}$.
}
\MNOTE[green]{Change B.24}

\begin{lem}	\label{red}
    Let $n \geq 2$ be an integer, and let $x \in [n]$.
    Let $p$ be a prime factor of $n$.
    \begin{enumerate}
%		\item If $n \not\equiv 2 \pmod 4$ and $\zeta_n^x \not\in \{ \zeta_{n/2}^j \mid j \in [n/2] \}$, then $-\zeta_n^x \in \{ \zeta_{n/2}^j \mid j \in [n/2] \}$.	\label{red:1}
	\item If $n \equiv 0 \pmod 4$, then 
        $$
        \pi^{(\TR{2}\TB{p}\MNOTE[green]{Change~B.24})}_n(-\zeta_n^x) = \begin{cases}
            \left( \pi^{(2)}_n(\zeta_n^x)+2 \right) \bmod 4 & \text{ if } p=2, \\
            \pi^{(p)}_n(\zeta_n^x) & \text{ otherwise,}
            \end{cases}
        $$
        and $\theta^{(p)}_n(-\zeta_n^x) = \theta^{(p)}_n(\zeta_n^x) $
	for every prime factor $p$ of $n$.	
	\label{red:1}
	\item If $n \equiv 0 \pmod 3$, then 	
		$
			\pi^{(3)}_n(\zeta_n^{-x}) = \begin{cases}
	   		\left( 3-\pi^{(3)}_n(\zeta_n^x) \right) \bmod 3 & \text{ if } \theta^{(3)}_m(\zeta_m^x) = 0, \\
				2-\pi^{(3)}_n(\zeta_n^x) & \text{ otherwise.}
			\end{cases}
		$	\label{red:2}
    \end{enumerate}
\end{lem}
\begin{proof}
%	We write $q_j := p_j^{f_j}$ for short.
%	We prove~\ref{red:1}.
%	Assume $n \not\equiv 2 \pmod 4$ and $\zeta_n^x \not\in \{ \zeta_{n/2}^j \mid j \in [n/2] \}$.
%	Then $p_1 = 2$, $f_1=1$ and $\Psi_n^{(2)}(x) = 1$.
%	Hence
%	\begin{align*}
%		-\zeta_n^x 
%		=  -\zeta_2^{\Psi_n^{(2)}(x)}\zeta_{p_2}^{\Psi_n^{(q_2)}(x)} \cdots \zeta_{q_t}^{\Psi_n^{(p_t)}(x)} 
%		=  \zeta_{q_2}^{\Psi_n^{(p_2)}(x)} \cdots \zeta_{q_t}^{\Psi_n^{(p_t)}(x)}
%		\in \{ \zeta_{n/2}^j \mid j \in [n/2] \}.
%	\end{align*}
%	
%	Next we prove~\ref{red:2}.
	Let $n = p_1^{f_1} p_2^{f_2} \cdots p_t^{f_t}$ be the prime factorization of $n$,
	and set $q_i := p_i^{f_i}$.
	Write $z := \zeta_n^x$ and $x_i := \Psi_n^{(p_i)}(z)$ for short.
	First we show~\ref{red:1}.
	Assume $n \equiv 0 \pmod 4$.
	Then we may assume $p_1=2$ and $f_1 \geq 2$.
	We have
	\begin{align*}
		-\zeta_n^x 
		&=  -\zeta_{q_1}^{x_1} \cdot \zeta_{q_2}^{x_2} \cdots \zeta_{q_t}^{x_3} 
		=  -\zeta_{4}^{\pi_n^{(2)}(z)}\zeta_{q_1}^{\theta_n^{(2)}(z)}  \zeta_{q_2}^{x_2} \cdots \zeta_{q_t}^{x_3} 
		=  \zeta_{4}^{\pi_n^{(2)}(z) + 2}\zeta_{q_1}^{\theta_n^{(2)}(z)} \zeta_{q_2}^{x_2} \cdots \zeta_{q_t}^{x_3} .
	\end{align*}
	Hence desired conclusion holds.
	
	Next we show~\ref{red:2}.
	Assume $n \equiv 0 \pmod 3$.
	Then we may assume $p_1 = 3$ and $f_1 \geq 1$.
	We have 
	\begin{align*}
		\TB{\zeta_{3^{f_1}}^{\Psi^{(3)}_n(z^{-1})}}\MNOTE[green]{Change~B.26}
            = \zeta_{3}^{\pi_n^{(3)}(z^{-1})}\zeta_{3^{f_1}}^{\theta_n^{(3)}(z^{-1})}
		= \zeta_{3}^{-\pi_n^{(3)}(z)}\zeta_{3^{f_1}}^{-\theta_n^{(3)}(z)}
		= \begin{cases}
			\zeta_{3}^{3-\pi_n^{(3)}(z)} & \text{ if } \theta^{(3)}_m(z) = 0, \\
			\zeta_{3}^{2-\pi_n^{(3)}(z)}\zeta_{3^{f_1}}^{3^{f_1-1}-\theta_n^{(3)}(z)}	& \text{ otherwise}.
		\end{cases}
	\end{align*}
	Hence the desired result holds.
\end{proof}

\subsection{Proof of Theorem~\ref{M1}}
By using Theorem~\ref{0117-1},
we can change integral bases of cyclotomic fields.
First, we provide one lemma as a situation where Theorem~\ref{0117-1} can be easily applied.

\begin{lem} \label{0210-1}
Let $m \geq 2$ be an integer with $m \not\equiv 2 \pmod 4$,
and let $S \subset \{1,2, \dots, m-1\}$.
If $|S| \leq \varphi(p^*)$ for every prime factor $p \neq 2$ of $m$,
then there exists an integral basis for $\MB{Q}(\zeta_m)$ containing either $\zeta_m^r$ or $-\zeta_m^r$ for each $r \in S$.
In particular,
if all $\pm \zeta_m^r$ $(r \in S)$ are pairwise-distinct, then
\begin{align}	\label{0210-1:1}
	 \frac{1}{2}\sum_{r \in S} \pm \zeta_m^r
\end{align}
is not an algebraic integer.
\end{lem}

\begin{proof}
Let $m = 2^{f_0} p_1^{f_1} p_2^{f_2} \cdots p_t^{f_t}$ be the prime factorization of $m$ where $f_0$ is a non-negative integer not equal to $1$ and $f_1, \ldots,  f_t$ are positive integers.

Assume $|S| \leq \varphi(p^*)$ for every prime factor $p \neq 2$.
We give an integral basis of $\MB{Q}(\zeta_m)$ containing either $\zeta_m^r$ or $-\zeta_m^r$.
For every $j \geq 1$, we let
$
	B_{p_j} = [p_j^{f_j-1}],
$
and can take $A_{p_j}$ by the assumption such that
\begin{align}	\label{0210-1:2}
	\{ \pi_m^{(p_j)} (\zeta_m^r) \mid r \in S \}
	\subset A_{p_j} \subset [p_j]
	\quad \text{and} \quad |A_{p_j}| = \varphi(p_j^*).
\end{align}
In addition, if $m \equiv 0 \pmod 4$, then we let
\begin{align}
	B_{2} = [{2^{f_0}}/{2^*} ] = [2^{f_0-2}] \ \text{ and } \ A_2 = \{0,1\}.	\label{0210-1:3}
\end{align}
By Theorem~\ref{0117-1}, we derive an integral basis $I$ in~\eqref{0117-1:1} of $\MB{Q}(\zeta_m)$.

Next, we show that for each $r \in S$, either $\zeta_{m}^r$ or $-\zeta_{m}^r$ is contained in $I$.
Since
\begin{align*}
	\zeta_{m}^r 
%	&= (\Phi_m \circ \Psi_m)(\zeta_{m}^r) \\
%	&= \Phi_m ( \Psi_m^{(p_0)}(\zeta_{m}^r), \Psi_m^{(p_1)}(\zeta_{m}^r), \dots, \Psi_m^{(p_s)}(\zeta_{m}^r) ) \\
%	&= \prod_{j=0}^t \zeta_{q_j}^{\Psi_m^{(p_j)}(\zeta_{m}^r)} \\
	&= \prod_{j=0}^t \zeta_{p_j^*}^{\pi_m^{(p_j)}(\zeta_{m}^r)} \zeta_{q_j}^{\theta_m^{(p_j)}(\zeta_{m}^r)},
\end{align*}
where $p_0 = 2$,
it suffices to verify that there exists $\varepsilon \in \{\pm 1\}$ such that for every $j \in [t]$, $\pi_m^{(p_j)}(\varepsilon \zeta_{m}^r) \in A_{p_j}$ and $\theta_m^{(p_j)}(\varepsilon \zeta_{m}^r) \in B_{p_j}$.

In the case of $f_0=0$, we see that all $\zeta_m^r$ is contained in $I$.
Hence we assume $f_0 \geq 2$.
If $\TR{\pi_m(2)}\TB{\pi_m^{(2)}}\MNOTE[green]{Change B.27}(\zeta_m^r) \in \{0,1\}$ for some $r \in S$,
then $\zeta_m^r \in I$.
Otherwise, we have  by Lemma~\ref{red}, $\pi^{(2)}_m(-\zeta_m^{r}) \equiv \pi^{(2)}_m(\zeta_m^{r})+2 \equiv 0,1 \pmod 4$
and $\pi_m^{(p_j)}(-\zeta_m^r) \in A_{p_j}$ for every $j \geq 1$.
These imply $-\zeta_m^r \in I$.
Thus, $I$ is a desired integral basis.

Lastly, if all $\pm \zeta_m^r$ $(r \in S)$ are pairwise-distinct, then $\zeta_m^r$ $(r \in S)$ are $\MB{Q}$-linearly independent.
Thus,~\eqref{0210-1:1} is not an algebraic integer.
\end{proof}

More precise discussion shows that $\Delta$ in~\eqref{0120-2}, which is our interest, is not an algebraic integer.

\begin{lem} \label{0120-1}
Let $m \geq 2$ be an integer with $m \not\equiv 2 \pmod{4}$,
and let $p, q$ be positive integers.
If $\zeta_m^{p}, \zeta_m^{-p}, \zeta_m^{q}, \zeta_m^{-q}$ are $\MB{Q}$-linearly independent,
then 
\[
\Delta = \frac{\pm \zeta_{m}^{p} \pm \zeta_{m}^{-p} \pm \zeta_{m}^{q} \pm \zeta_{m}^{-q}}{2}
\]
is not an algebraic integer.
\end{lem}

\begin{proof}%[{\bf Proof of Lemma~\ref{0120-1}}]
Let $S = \{\pm p, \pm q\}$.
Let $m = 2^{f_0} 3^{f_1} p_2^{f_2} p_3^{f_3} \cdots p_t^{f_t}$ be the prime factorization of $m$ where $f_0$ and $f_1$ are non-negative integers and $f_2, \ldots,  f_t$ are positive integers.

First, \TB{since $2 = \varphi(3) < 4 \leq \varphi(p)$ for any prime number $p > 3$,}\MNOTE[green]{Change B.28} we consider the case where $m$ is not divisible by $3$
or where $|\{\pi_m^{(3)} (\zeta_m^r) \mid r \in S \}| \leq \varphi(3)$.
In this case, Lemma~\ref{0210-1} is applicable,
and hence we see that $\Delta$ is not an algebraic integer.
Below, we consider the case where $m$ is divisible by $3$, i.e., $f_2 \geq 1$,
and $|\{\pi_m^{(3)} (\zeta_m^r) \mid r \in S \}| = 3$.
Let $p_1 := 3$.
Let $B_3 := [3^{f_1-1}]$ and $A_3 := \{1,2\}$.
For every $j \geq 2$, we define
$
	B_{p_j} = [p_j^{f_j-1}],
$
and $A_{p_j}$ as in~\eqref{0210-1:2}.
In addition, if $m \equiv 0 \pmod 4$, then we define $A_2$ and $B_2$ by~\eqref{0210-1:3}.
By Theorem~\ref{0117-1}, we derive an integral basis $I$ in~\eqref{0117-1:1} of $\MB{Q}(\zeta_m)$.

Since $\{\pi_m^{(3)} (\zeta_m^r) \mid r \in S \} = \{0,1,2\}$,
there exists $r' \in S$ such that $\pi_m^{(3)} (\zeta_m^{r'}) = 0$.
We may assume that $r' = p$, i.e.,
$
\pi_m^{(3)} (\zeta_m^{p}) = 0
$
without loss of generality.
We have
$\zeta_m^p = -\zeta_3^1\zeta_m^p-\zeta_3^2\zeta_m^p$, 
$\pi_m^{(3)}(\zeta_3^1\zeta_m^p)=1$
and
$\pi_m^{(3)}(\zeta_3^2\zeta_m^p)=2.$
Also, we have by Lemma~\ref{red},
\begin{equation} \label{0123-3}
\pi_m^{(3)} (\zeta_m^{-p}) \in \{0,2\}.
\end{equation}
Hence, either $\pi_m^{(3)} (\zeta_m^{q})$ or $\pi_m^{(3)} (\zeta_m^{-q})$ must be $1$.
By Lemma~\ref{red}, we may assume that 
$$\pi_m^{(3)} (\zeta_m^{q}) = 1\text{ and } \pi_m^{(3)} (\zeta_m^{-q}) \in \{1,2\}$$ without loss of generality.

Assume $\pi^{(3)}(\zeta_m^{-p}) = 2$.
Take an arbitrary element $z \in \{\zeta_3^1\zeta_m^p, \zeta_3^2\zeta_m^p, \zeta_m^{-p}, \zeta_m^{\pm q}\}$.
Then, we have by Lemma~\ref{red}, 
\begin{align}	\label{0123-5}
	\pi_m^{(p_i)}(z) = \pi_m^{(p_i)}(-z)  \in A_{p_i} \quad \text{ and } \quad \theta_m^{(p_i)}(z) = \theta_m^{(p_i)}(-z) \in B_{p_i}
\end{align}
for every $1 \leq i \leq t$.
Furthermore, we have by Lemma~\ref{red} again,
\begin{align}	\label{0123-6}
	\pi_m^{(2)}(z) \text{ or }  \pi_m^{(2)}(-z)  \in A_2 \quad \text{ and } \quad \theta_m^{(2)}(z), \theta_m^{(2)}(-z) \in B_2
\end{align}
if $f_0 \geq 2$.
Hence, $\zeta_3^1\zeta_m^p, \zeta_3^2\zeta_m^p, \zeta_m^{-p}, \zeta_m^{\pm q}$ are contained in $I$, up to multiplication by $\pm 1$.
Since $\MB{Q}$-linear independence of $\zeta_m^{-p}, \zeta_m^{\pm q}$ implies that $\zeta_m^{-p}, \zeta_m^{\pm q}$ are pairwise distinct up to multiplication by $\pm 1$,
we see
\begin{align*}
	\Delta = \frac{\pm \zeta_{m}^{p} \pm \zeta_{m}^{-p} \pm \zeta_{m}^{q} \pm \zeta_{m}^{-q}}{2} 
	=\frac{\mp \zeta_3^1\zeta_{m}^{p} \mp \zeta_3^2\zeta_{m}^{p} \pm \zeta_{m}^{-p} \pm \zeta_{m}^{q} \pm \zeta_{m}^{-q}}{2} 
\end{align*}
is not an algebraic integer.

Assume $\pi^{(3)}(\zeta_m^{-p}) = 0$.
Then,
\begin{align*}
	\zeta_m^{-p} = -\zeta_3^1\zeta_m^{-p}-\zeta_3^2\zeta_m^{-p}, 
	\quad \pi_m^{(3)}(\zeta_3^1\zeta_m^{-p})=1
	\quad \text{ and }
	\quad \pi_m^{(3)}(\zeta_3^2\zeta_m^{-p})=2.
\end{align*}
We take an arbitrary element $z \in \{\zeta_3^1\zeta_m^{\pm p}, \zeta_3^2\zeta_m^{ \pm p}, \zeta_m^{\pm q}\}$.
Then, by Lemma~\ref{red}, we have~\eqref{0123-5} for every $1 \leq i \leq t$, and obtain~\eqref{0123-6} if $f_0 \geq 2$.
Thus, $\zeta_3^1\zeta_m^{\pm p}, \zeta_3^2\zeta_m^{ \pm p}, \zeta_m^{\pm q}$ are contained in $I$, up to multiplying by $\pm 1$.
Also, $\pm \zeta_3^1\zeta_m^{\pm p}, \pm \zeta_3^2\zeta_m^{ \pm p}$ are pairwise distinct by the assumption that $\zeta_m^{\pm p}$ are $\MB{Q}$-linearly independent.
Therefore,
\begin{align*}
	\Delta = \frac{\pm \zeta_{m}^{p} \pm \zeta_{m}^{-p} \pm \zeta_{m}^{q} \pm \zeta_{m}^{-q}}{2} 
	=\frac{\mp \zeta_3^1\zeta_{m}^{p} \mp \zeta_3^2\zeta_{m}^{p} \mp \zeta_3^1\zeta_{m}^{-p} \mp \zeta_3^2\zeta_{m}^{-p} \pm \zeta_{m}^{q} \pm \zeta_{m}^{-q}}{2} 
\end{align*}
is not an algebraic integer.
\end{proof}

We are now almost ready to prove our second main result (Theorem~\ref{M1}).
We show the case where $l$ is odd first,
and later we deal with the case where $l$ is even.

\begin{thm}
Let $X=X(\MB{Z}_{2l}, \{\pm a, \pm b\})$ be a connected $4$-regular circulant graph,
and let $a+b \neq l$.
If $l$ is odd,
then $X$ does not admit perfect state transfer between vertex type states.
\end{thm}

\begin{proof}
Let $l$ be an odd integer.
By Lemma~\ref{0114-1}, it suffices to prove 
\[ \Delta := \frac{(-1)^a \zeta_{l}^{\varepsilon(a)} + (-1)^a \zeta_{l}^{-\varepsilon(a)} + (-1)^b \zeta_{l}^{\varepsilon(b)} + (-1)^b \zeta_{l}^{-\varepsilon(b)} }{2}
\]
is not an algebraic integer.
To confirm that the assumptions of Lemma~\ref{0120-1} are satisfied,
we want to show that $4\varepsilon(a), 4\varepsilon(b), 2(\varepsilon(a) \pm \varepsilon(b)) \not\equiv 0 \pmod{l}$ \TB{as these will be needed in order to apply Theorem~\ref{0127-1}}\MNOTE{Change A.17}.
This can be proved by dividing the case by the parity of $a$ and $b$.
Since tedious calculations follow,
we show only the case $(\varepsilon(a), \varepsilon(b)) = (\frac{a}{2}, \frac{b+l}{2})$ as a representative.
Suppose $4\varepsilon(a) \equiv 0 \pmod{l}$.
We have $2a \equiv 0 \pmod{l}$.
Since $l$ is odd, $a \equiv 0 \pmod{l}$, which contradicts~(\ref{0128-2}).
%Thus $4\varepsilon(a) \not\equiv 0 \pmod{l}$.
Suppose $4\varepsilon(b) \equiv 0 \pmod{l}$.
We have $0 \equiv 2(b+l) \equiv 2b \pmod{l}$.
Since $l$ is odd, $b \equiv 0 \pmod{l}$, which contradicts~(\ref{0128-2}).
Suppose $2(\varepsilon(a) + \varepsilon(b)) \equiv 0 \pmod{l}$.
We have $0 \equiv a + b+l \equiv a + b \pmod{l}$,
so $a + b \equiv 0, l \pmod{2l}$.
Both contradict (\ref{0128-3}) or our assumption.
Suppose $2(\varepsilon(a) - \varepsilon(b)) \equiv 0 \pmod{l}$.
We have $0 \equiv a - b - l \equiv a - b \pmod{l}$,
so $a - b \equiv 0, l \pmod{2l}$.
Both contradict (\ref{0128-3}) or (\ref{0128-4}).
Therefore by Lemma~\ref{0127-1} and Lemma~\ref{0120-1},
we see that $\Delta$ is not an algebraic integer.
\end{proof}

The reason why the above proof worked is that the assumption that $l$ is odd guaranteed $\MB{Q}$-linear independence of $\zeta_l^{\varepsilon(r)}$ $(r \in S=\{\pm a,\pm b\})$. 
However, when $l$ is even,
for example $l=6$ and $(a,b)=(1,3)$,
then it holds that $\zeta_{2l}^{-b} = \zeta_4^3 = -\zeta_4^1 = -\zeta_{2l}^{b}$.
This means that the case where $l$ is even does not necessarily guarantee $\MB{Q}$-linear independence of $\zeta_{2l}^{r}$ $(r \in S)$. 
Therefore, more devised discussion is required when $l$ is even.

\begin{lem} \label{0228-2}
Let $X=X(\MB{Z}_{2l}, \{\pm a, \pm b\})$ be a connected $4$-regular graph.
If $4a \equiv 0 \pmod{2l}$,
then $4b \not\equiv 0 \pmod{2l}$.
Moreover,
$\zeta_{2l}^{b}$ and  $\zeta_{2l}^{-b}$ are $\MB{Q}$-linearly independent.
\end{lem}

\begin{proof}
Our assumption derives that $2a \equiv 0 \pmod{l}$,
so $2a \equiv 0, l \pmod{2l}$.
By~(\ref{0128-2}), we have $2a \equiv l \pmod{2l}$.
Next, we suppose that $4b \equiv 0 \pmod{2l}$.
Similarly, we obtain $2b \equiv l \pmod{2l}$.
This implies that $a-b \equiv 0, l \pmod{2l}$.
However, this contradicts~(\ref{0128-3}) or~(\ref{0128-4}).

Next,
we assume that a $\MB{Q}$-linear combination
\begin{equation} \label{0228-1}
c_1\zeta_{2l}^{b} + c_2\zeta_{2l}^{-b} = 0.
\end{equation}
Take the conjugate of both sides and add it to Equality~(\ref{0228-1}).
We have $2 (c_1 + c_2) \cos \frac{2b}{2l}\pi = 0$,
and hence $c_1 + c_2 = 0$ since $4b \not\equiv 0 \pmod{2l}$.
Substituting it into Equality~(\ref{0228-1}).
Then we have $2i c_1 \sin \frac{2b}{2l}\pi = 0$,
and hence $c_1 = 0$ since $4b \not\equiv 0 \pmod{2l}$.
Therefore, $\zeta_{2l}^{b}$ and  $\zeta_{2l}^{-b}$ are $\MB{Q}$-linearly independent.
\end{proof}

\begin{thm}
Let $X=X(\MB{Z}_{2l}, \{\pm a, \pm b\})$ be a connected $4$-regular with $a+b \neq l$.
If $l$ is even,
then $X$ does not admit perfect state transfer between vertex type states.
\end{thm}

\begin{proof}
%Next, we consider the case where $l$ is even.
Let $\Delta$ be the element of $\MB{Q}(\zeta_{2l})$ in~\eqref{nonalg:1}.
By Lemma~\ref{nonalg}, it suffices to check that $\Delta$ is not an algebraic integer.
First, we consider the case where $4a, 4b \not\equiv 0 \pmod{2l}$.
Then we claim that $2(a \pm b) \not\equiv 0 \pmod{2l}$.
Indeed, we suppose that $2(a \pm b) \equiv 0 \pmod{2l}$.
Then we have $a \pm b \equiv 0, l \pmod{2l}$,
which contradicts (\ref{0128-3}), (\ref{0128-4}), or our assumption.
Therefore by Lemma~\ref{0127-1} and Lemma~\ref{0120-1},
we see that $\Delta$ is not an algebraic integer.
Hence Lemma~\ref{0128-5} implies that the graph $X$ does not admit perfect state transfer.

Next,
we consider the case where either $4a \equiv 0 \pmod{2l}$ or $4b \equiv 0 \pmod{2l}$.
We may assume $4a \equiv 0 \pmod{2l}$ without loss of generality.
By Lemma~\ref{0228-2},
we have $4b \not\equiv 0 \pmod{2l}$.
Since $4a \equiv 0 \pmod{2l}$ and~(\ref{0128-2}),
we have $2a \equiv l \pmod{2l}$.
Thus, there exists an integer $s$ such that $2a - l = 2sl$,
i.e., $2a= (2s+1)l$.
We have $\zeta_{2l}^{\pm a} =\zeta_{4l}^{\pm 2a} = \zeta_{4l}^{\pm (2s+1)l} = \zeta_{4}^{\pm (2s+1)}$.
In particular, $\zeta_{2l}^{a} + \zeta_{2l}^{-a} = 0$.
Hence,
\[ \Delta = \frac{\zeta_{2l}^a + \zeta_{2l}^{-a} + \zeta_{2l}^b + \zeta_{2l}^{-b}}{2} = \frac{\zeta_{2l}^b + \zeta_{2l}^{-b}}{2}. \]
By Lemma~\ref{0228-2},
$\zeta_{2l}^{b}$ and  $\zeta_{2l}^{-b}$ are $\MB{Q}$-linearly independent, so Lemma~\ref{0210-1} derives that $\Delta$ is not an algebraic integer.
Therefore, Lemma~\ref{0128-5} implies that the graph $X$ does not admit perfect state transfer.
\end{proof}

\section*{Acknowledgements} \MNOTE[black]{Change C.5}
\TB{We are deeply grateful to the anonymous reviewers for their insightful feedback and constructive suggestions,
which have significantly contributed to improving the manuscript.
We would also like to thank Akihiro Higashitani for his helpful advice and insightful discussions.}
S.K. is supported by JSPS KAKENHI (Grant Number JP20J01175).
K.Y. is supported by JSPS KAKENHI (Grant Number JP21J14427).

\bibliographystyle{plain}
\bibliography{references}
%automorphism
%circulant
\end{document}